\documentclass[10pt]{amsart}
\usepackage[margin=1.25in]{geometry}
\usepackage{amsmath,amsthm,amssymb,amsfonts}
\usepackage{subcaption}
\usepackage[countmax]{subfloat}
\usepackage[author={Kai Nakamura}]{pdfcomment} 
\usepackage[pdfencoding=auto]{hyperref} 
\usepackage{xcolor}

\usepackage{todonotes}

\makeatletter
\providecommand\@dotsep{5}
\renewcommand{\listoftodos}[1][\@todonotes@todolistname]{%
  \@starttoc{tdo}{#1}}
\makeatother

\usepackage{float}
\usepackage{comment}

\usepackage{graphicx}
\graphicspath{{images/}}

\usepackage{setspace} 

\newcommand{\Z}{\mathbb{Z}}

\newcommand{\CP}{\mathbb{CP}^2}

\newcommand{\nCPbar}{\# n \overline{\mathbb{CP}}^2}
\newcommand{\CPbar}{\overline{\mathbb{CP}}^2}

\newcommand{\zsg}{\phi : S^3_0(K) \rightarrow S^3_0(K')}

\def\inter{\mathop{\rm int}}
\newcommand{\torusnbhd}{S^1 \times S^1 \times D^2}

\newcommand{\anntwifam}{\phi_k:S^3_0(J_0) \rightarrow S^3_0(J_k)}
\newcommand{\gsh}[1]{g_{sh}^0( #1 )}

\usepackage[backend=biber,style=alphabetic,citestyle=alphabetic,maxnames=10,maxalphanames=99]{biblatex}

\renewbibmacro{in:}{%
  \ifentrytype{article}
    {}
    {\bibstring{in}%
     \printunit{\intitlepunct}}}

\usepackage{hyperref}

\newtheorem{thm}{Theorem}[section]
\newtheorem{lem}[thm]{Lemma}
\newtheorem{cor}[thm]{Corollary}
\newtheorem{mydef}[thm]{Definition}
\newtheorem{prop}[thm]{Proposition}
\newtheorem{remark}[thm]{Remark}

\newtheorem{conj}[thm]{Conjecture}
\newtheorem{prob}[thm]{Problem}

\title {Torus surgery on Knot Traces}
\author{Kai Nakamura}
\address{  Stanford University\\
      Department of Mathematics\\
      Building 380, Stanford, California 94305,
   USA}
\email{kainaka@stanford.edu}
\date{\today}

\begin{document}

\begin{abstract}
We initiate the study of torus surgeries on knot traces.
Our key technical insight is realizing the annulus twisting construction of Osoinach as a torus surgery on a knot trace.
We present several applications of this idea.
We find exotic elliptic surfaces that can be realized as surgery on null-homologously embedded traces in a manner similar to that proposed by Manolescu and Piccirillo.
Then we exhibit exotic traces with novel properties and improve upon the known geography for exotic Stein fillings.
Finally, we construct new potential counterexamples to the smooth $4$-dimensional Poincar\'e Conjecture.
\end{abstract}
\maketitle

\section{Introduction}
%
%
This paper has a simple thesis statement: torus surgeries are a powerful tool to construct and study $4$-manifolds and we can apply this technology to knot traces.
Recall that the zero trace $X_0(K)$ of a knot $K$ is a $4$-manifold formed from attaching a zero framed $2$-handle to the $4$-ball.
This approach echoes that of Yasui's cork twisting of knot traces \cite{yasui_akb_kirby_conj}.
There he shows how to construct pairs of knot traces that differ by a cork twist.
This had important applications such as resolving the Akbulut-Kirby Conjecture and constructing new families of exotic traces.

A torus surgery removes a $T^2 \times D^2$ from a $4$-manifold and glues it back in by some homeomorphism of the boundary $3$-torus.
Almost all constructions of exotic $4$-manifolds can be realized as torus surgeries \cite{round_handles}.
Therefore, we would like to apply torus surgeries to construct and study knot traces.
To do this, we will make use of Osoinach's annulus twisting construction \cite{ann_twist}.
This is a method for producing infinite collections of knots $\{J_k\}_{k \in \Z}$ that share a zero surgery $\anntwifam{}$.
The key technical theorem underlying our results is the following relationship between annulus twisting and torus surgery.
\begin{thm}
An annulus twist homeomorphism $\anntwifam$ can be realized $4$-dimensionally as a torus surgery of $X_0(J_0)$ resulting in $X_0(J_k)$.
\label{thm:anntwivague}
\end{thm}
See Theorem \ref{thm:ann_twi_log_trans} in Section \ref{sec:twist} for a more precise statement.
We can now use what we know about torus surgeries on general $4$-manifolds and apply it to knot traces.
We present these applications as a sequence of short vignettes that we summarize below.
\subsection{Zero surgery homeomorphisms and exotic 4-manifolds}
After Perelman\cite{perelman2002entropy,perelman2003ricci}, the most notable remaining case of the generalized Poincar\'e Conjecture is the smooth $4$-dimensional Poincar\'e Conjecture (SPC4).
Long open, SPC4 asserts there are no exotic $4$-spheres: $4$-manifolds homeomorphic, but not diffeomorphic to the standard $4$-sphere.
One approach initiated by Manolescu and Piccirillo is to use sliceness and zero surgery homeomorphisms to construct promising homotopy $4$-spheres \cite{zero_surg_exotic}.
Recall that a knot $K$ is slice if there is a smooth, properly embedded disk $D$ in $B^4$ with boundary $K$.
If a knot $K$ is slice and there is a zero surgery homeomorphism $\zsg$, then one can construct a homotopy $4$-sphere.
It is not immediately clear if these homotopy $4$-spheres are standard.
Using annulus twists of the ribbon knot $8_8$, Manolescu and Piccirillo constructed an infinite family of homotopy $4$-spheres $Z_k$.
The author showed these $Z_k$ were standard.
\begin{thm}[Theorem 1.3 of \cite{trace_emb_zero}]
The Manolescu-Piccirillo homotopy $4$-spheres $Z_k$ constructed by annulus twists of the ribbon knot $8_8$ are standard.
\end{thm}
We reexamine these homotopy $4$-spheres using Theorem \ref{thm:anntwivague} to view them as torus surgeries instead.
This allows us to give a new proof that these $Z_k$ are standard.
One of the main purposes of this article is to contrapose this result and provide evidence for the viability of this approach.
Unfortunately, constructing an exotic $4$-sphere is a notoriously difficult problem.
We instead retreat to the setting of elliptic surfaces; these are fundamental examples of $4$-manifolds and much more amenable to exotic constructions than the $4$-sphere.
Since we are no longer in the $4$-sphere, we need an appropriate notion of sliceness.
For a smooth $4$-manifold $X$, let $X^\circ = X - \inter(B^4)$ and we say a knot $K$ is $H$-slice in $X$ if it bounds a smoothly embedded nullhomologous disk in $X^\circ$.
Now using the plethora of exotica constructed by torus surgeries on elliptic surfaces and Theorem \ref{thm:anntwivague}, we can adapt the Manolescu-Piccirillo construction to elliptic surfaces.
\begin{thm}
There exists a knot $H$-slice in an elliptic surface such that annulus twisting can give rise to infinitely many exotic elliptic surfaces.
\end{thm}
This provides evidence for the viability of Manolescu and Piccirillo's approach to disproving SPC4.
If we wish to construct an exotic $4$-sphere, then we should first be able to construct exotic elliptic surfaces.
Moreover, along the way we prove the following result.
\begin{thm}
The $-5_2$ knot is $H$-slice in the $K3$ surface with zero surgery homeomorphisms $\phi_k: S^3_0(-5_2) \rightarrow S^3_0(-5_2)$ that give rise to infinitely many exotic $K3$s.
Moreover, the $-5_2$ knot is the smallest crossing number non-trivial knot that is $H$-slice in $K3$ and is the only non-slice prime knot with six or fewer crossings that is $H$-slice in $K3$.
\end{thm}
This theorem successfully bridges the classification of slice knots to exotic smooth structures on $4$-manifolds.
It shows that as you classify $H$-slice knots in a $4$-manifold, you would naturally arrive at constructing exotic $4$-manifolds via $H$-sliceness and zero surgery homeomorphisms.
\subsection{Exotic traces with novel properties}
Since $4$-dimensional $1$ and $3$-handles attach uniquely, all of the interesting topology of exotic $4$-manifolds is contained in the $2$-handles.
Therefore, knot traces provide a setting in miniature to study exotic $4$-manifolds.
Moreover, a knot's trace is intimately connected to its slice properties via the trace embedding lemma.
Akbulut constructed the first examples of exotic traces with non-zero framing \cite{ex_trace_akb} and Yasui gave the first examples of exotic zero traces \cite{yasui_akb_kirby_conj}.
These and most other examples of exotic traces were distinguished by the shake genus.
Let us recall its definition.
\begin{mydef}
For a knot $K$, the zero shake genus $\gsh{K}$ is the minimal genus of a surface generating the second homology of $X_0(K)$.
Moreover, if $\gsh{K} = 0$, then $K$ is called zero shake slice.
\end{mydef}
One of the primary methods to smoothly distinguish a pair of homeomorphic $4$-manifolds is to show that they have different minimal genus functions.
So it would be interesting to find exotic $4$-manifolds which are not distinguished by minimal genera.
In the setting of knot traces, the problem becomes finding exotic knot traces with the same shake genus.
There is additional value in moving past the shake genus to distinguish exotic knot traces.
It is impossible to exhibit exotic knot traces with several interesting properties using the shake genus.
For example, the shake genus does not distinguish the traces of  infinitely many knots, knots related by annulus twisting, or concordant knots.
\begin{itemize}
    \item Infinitely many knots:
    An infinite family of zero surgery homeomorphisms $\phi_k:S^3_0(K_0) \rightarrow S^3_0(K_k)$ with $k \in \Z$ puts a finite bound on $\gsh{K_k}$.
    Let $F \subset S^3_0(K_0)$ be a surface representing a generator of $H_2(S^3_0(K_0))$, e.g. a capped off Seifert surface of $K_0$.
    Then $\phi_k(F)$ in $S^3_0(K_k)$ is mapped by the boundary inclusion $S^3_0(K_k) = \partial X_0(K_k)$ to a generator of $H_2(X_0(K_k))$.
    The zero shake genus of $\{K_k\}_{k \in \Z}$ is bounded by $g(F)$ and can only attain finitely many values.
    Therefore, the shake genus can only distinguish finitely many zero traces.
    \item Annulus twisting: When $\{K_k\}_{k \in \Z}$ are related by annulus twisting, the situation is even worse.
    From the definition of annulus twisting, it is easy to see a genus one surface representing a generator of $H_2(X_0(K_k))$ and therefore such $K_k$ has bounded zero shake genus $g_{sh}^0(K_k) \le 1$.
    In theory, a pair of knots $K$ and $K'$ related by annulus twisting could have exotic zero traces distinguished by zero shake genera $\gsh{K} = 0$ and $\gsh{K'} = 1$.
    However, this would also disprove long open and difficult problems depending on if $K$ is slice.
    \begin{itemize}
        \item If $K$ is slice,
        then $S^3_0(K) \cong S^3_0(K')$ would be a homeomorphism between a slice knot $K$ and non-slice knot $K'$.
        This would imply there is an exotic $4$-sphere as discussed in Section \ref{sec:MP_spheres}.
        \item If $K$ is not slice,
        then it would be the first example of non slice zero shake slice knot.
        It is known that there is an $r$-shake slice knot that is not slice for any non-zero $r \in \Z$ \cite{lickorish,shake_slice_and_shake_conc,twodim,shakeslice_akb}.
        The remaining $r = 0$ case is long open.
    \end{itemize}
    \item Concordant knots: 
    It is well known that concordant knots have the same slice genus, moreover, they have the same shake genus.
    If there is a concordance $C \subset S^3 \times I$ from $K'$ to $K$, one can construct a trace embedding $X_0(K') \subset X_0(K)$.
    Attach a $2$-handle to $S^3 \times I$ along $K$ and fill $S^3 \times 0$ with a ball $B$ to get $X_0(K)$.
    Use the core of the $2$-handle and the concordance $C$ to get a $2$-handle attached to $B$ along $K'$.
    The embedding $X_0(K') \subset X_0(K)$ implies that $g_0^{sh}(K') \le g_0^{sh}(K)$ and the symmetry gives equality.
    Therefore, the shake genus is useless for distinguishing exotic traces of concordant knots.
    This argument also implies that the knots' traces have the same genus function.
\end{itemize}

We overcome these difficulties and construct a family of exotic traces with several novel properties.
\begin{thm}
\label{thm:ex_trace}
There exists a family of knots with mutually exotic traces that
\begin{enumerate}
    \item Consists of infinitely many knots
    \item Related by annulus twisting
    \item Have the same slice genus
    \item Concordant to each other
    \item Have the same zero shake genus
    \item The knots' traces have the same genus function
\end{enumerate}
\end{thm}
Previously, there were no examples of exotic traces with \textit{any} of the listed properties, whereas this family has \textit{all} of them.
Exotic traces with these properties were inaccessible because the shake genus was the main method to distinguish exotic traces.
By using Theorem \ref{thm:anntwivague} to relate these traces to exotic elliptic surfaces constructed by torus surgeries, we are able to move past this difficulty and construct such traces.

Yasui constructed exotic pairs of knot traces as the Stein traces of a pair of Legendrian knots and used the adjunction inequality for Stein surfaces to distinguish them \cite{yasui_akb_kirby_conj}.
Despite being a pair of exotic Stein traces, it is not clear that they have the same contact boundary.
Here we give the first example of Legendrian knots with exotic Stein traces with their contact boundaries identified.
This gives a truly symplectic analog of the phenomena of exotic traces.

\begin{thm}
There exists an infinite family of Legendrian knots related by contact annulus twisting whose Stein traces are homeomorphic, but non-diffeomorphic Stein fillings of the same contact $3$-manifold.
\end{thm}
With this we immediately obtain the following corollary.
\begin{cor}
There exist infinitely many homeomorphic, but non-diffeomorphic Stein surfaces with $b_2 =1$ that are Stein fillings of the same contact $3$-manifold.
\end{cor}
Exotic Stein fillings are candidates for cut and paste operations to construct exotic symplectic $4$-manifolds.
Furthermore, exotic Stein fillings can be thought of as a boundary analog of the problem of constructing exotic symplectic $4$-manifolds.
Smaller exotic Stein fillings would then correspond analogously to small exotic symplectic $4$-manifolds.
Akhmedov-Etnyre-Mark-Smith constructed the first example of exotic Stein fillings of a contact $3$-manifold \cite{note_stein_fill}
The size was whittled down by Akbulut-Yasui who had an example with $b_2 = 2$ \cite{akb_yas_stein}.

\subsection{New potential counterexamples to SPC4}
Manolescu and Piccirillo found five knots such that if any were slice, then an exotic $S^4$ would exist \cite{zero_surg_exotic}.
One would be able to construct an exotic $S^4$ by using a zero surgery homeomorphism with a non-slice knot.
The author then showed that these knots were not slice by using blowups of the associated knot traces to stably relate them to a non-slice knot.
Moreover, the author generalized these techniques to the entire family of zero surgery homeomorphisms considered by Manolescu and Piccirillo \cite{trace_emb_zero}.
This left a gap of potential counterexamples to SPC4.
There the author gave some examples of zero surgery homeomorphisms where his techniques did not apply.
These were a halfhearted attempt, they were more to illustrate that the Manolescu-Piccirillo approach was not yet dead.

Since the author's work on the Manolescu-Piccirillo project, there have been no promising potential counterexamples to SPC4 in the same manner as the five Manolescu-Piccirillo knots.
We now remedy this by using the insights provided by adapting the Manolescu-Piccirillo construction to elliptic surfaces.
\begin{thm}
There is an infinite family of $4$-manifolds $\{C_k\}_{k \in \Z}$ such that if any embed in the $4$-sphere, then an exotic $4$-sphere exists.
\label{thm:ex4sphere}
\end{thm}
Manolescu and Piccirillo hoped to construct an exotic $4$-sphere by removing the trace of a slice knot from $S^4$ and replace it with the trace of a non-slice knot.
The result of this hypothesized surgery would be an exotic $4$-sphere.
Instead, we propose to find a $4$-manifold $C_k$ embedded in $S^4$ and replace it with the trace of the Conway knot.
Piccirillo showed that the Conway knot is not slice, but her proof left open whether the Conway knot is exotically slice \cite{Conway_knot}.
The $4$-manifolds $C_k$ of Theorem \ref{thm:ex4sphere} are formed from simply connected torus surgeries on the zero trace of the Conway knot.
These were found by careful reexamination of Piccirillo's proof and reinterpreting it as a torus surgery.
The author feel that this new family of potential counterexamples to SPC4 are more plausible than the Manolescu-Piccirillo knots.
First, such an exotic $4$-sphere would be detected by the non-sliceness of the Conway knot.
The Conway knot is the smallest knot that is topologically slice, not smoothly slice, and unknown if slice in a homotopy $4$-sphere.
This makes it a natural candidate to be exotically slice.
Moreover, such an exotic $4$-sphere would naturally be a torus surgery which most constructions of exotic $4$-manifolds can be realized as.
By combining the concordance approach to SPC4 with torus surgeries, this should be a more plausible way to construct an exotic $4$-sphere.

\subsection*{Conventions}
All manifolds are smooth and oriented.
Any embeddings or homeomorphisms are orientation preserving.
Boundaries are oriented with outward normal first.
All homology groups have integral coefficients.
For a smooth manifold $M$ with a submanifold $N$, let $\nu(N)\subset M$ denote a tubular neighborhood of $N$ and let $E(N)$ denote the exterior $M - \nu(N)$ of $N$.

\subsection*{Acknowledgments}
The author would like to thank Ciprian Manolescu and Lisa Piccirillo for helpful correspondences.
The author would also like to thank his advisors Bob Gompf and John Luecke for their help and support.
The author greatly appreciates help from Maggie Miller and Charles Stine.
This research was supported in part by NSF grants DMS-1937215 and DMS-2402259.

\section{Annulus twists as torus surgeries}
\label{sec:twist}
The constructions of this paper will employ a delicate interplay of three flavors of twists that we define here: torus surgeries in $4$-manifolds, then torus twists on $3$-manifolds, and then the annulus twist construction of zero surgery homeomorphisms.
The main result of this section is to show that the annulus twist construction can be interpreted as a torus surgery of a knot trace.
We present these in this order to match how they appear in Section \ref{sec:constr}.

The main construction of exotica begins with a family of exotic $4$-manifolds related by surgery on a nullhomologous torus.
From these nullhomologous tori we build our embedded nullhomologous traces.
The cut and paste of this surgery operation will then induce homeomorphisms of the zero surgery boundaries of these zero traces.
The relevant surgery is a torus surgery or logarithmic transform which is an operation on a torus $T$ in a $4$-manifold $X$ with self intersection number zero $[T] \cdot [T] = 0$.
Then $T$ has normal disk bundle $T \times D^2$ which has boundary a $3$-torus $T \times S^1$.
\begin{mydef}
The torus surgery on $T$ with a homeomorphism $\phi: T \times S^1 \rightarrow \partial E(T)$ is the $4$-manifold $X_{T,\phi} = E(T) \cup_\phi T \times D^2$ obtained by cutting out $T \times D^2$ from $X$ and regluing it by $\phi$.
\end{mydef}
Note that the torus surgery is determined by where $\phi$ maps the circle $\{pt\} \times \partial D^2$.
Pick a convenient basis $[\alpha],[\beta]$ for $H_1(T)$ and let $[\phi(\{pt\} \times \partial D^2)] = p[\{pt\} \times \partial D^2] + q[\alpha] + r[\beta]$.
This triple $p,q,r$ determines the torus surgery of $T$ and we may write $X_{p,q,r}(T)$.
However, we will will prefer to explicitly define the homeomorphism used for the torus surgery that realizes an annulus twist.
The homeomorphism relevant for us will be a torus twist on a torus in $T \times S^1$.
Let $S$ be a torus embedded in a $3$-manifold $M$ and identify a tubular neighborhood of $S$ as $S^1 \times S^1 \times I$.
Let $\alpha$ and $\beta$ denote $S^1 \times \{\zeta\} \times 0$ and $\zeta \times \{S^1\} \times 0$ for some $\zeta \in S^1$.
\begin{mydef}
The torus twist $\tau$ on $S$ parallel to $\alpha$ is the homeomorphism obtained from $\tau(\theta, \theta',t) = (\theta + 2\pi t,\theta', t)$ by extending as the identity on the rest of $M$ and smoothing.
\end{mydef}
Our main trick will be to decompose a torus twist into a pair of \textit{annulus twists}.
Suppose $A$ is an annulus $S^1 \times I$ embedded in a $3$-manifold $M$.
Let $\ell_0 \cup \ell_1 = S^1 \times \{0\} \cup S^1 \times \{1\}$ be the boundary of $A$ with orientation and framing induced by $A$.
Note that this orients $\ell_0$ and $\ell_1$ oppositely.
\begin{lem}
There is a homeomorphism $f_k:M \cong M_{1/k,-1/k}(\ell_0,\ell_1)$ called annulus twisting $A$.
\end{lem}

\begin{proof}
In $E(\ell_1 \cup \ell_2) = M -\nu(\ell_0 \cup \ell_1)$, $A$ restricts to a properly embedded annulus $A^*$ with a product neighborhood $A^* \times I$ that we parametrize as $S^1 \times I \times I$.
We define the following map $f_k$ on $A^* \times I$ which is just a Dehn twist on each $S^1 \times \{x\} \times I$ factor.
\[f_k: S^1 \times I \times I \rightarrow S^1 \times I \times I,\ f_k(\theta,x,t) = (\theta + k (2\pi)t,x,t)\]
On $A^* \times \partial I = A^* \times \{0,1\}$, $f_k$ is the identity and so extends to a map on $E(\ell_0 \cup \ell_1)$ which we continue to call $f_k$.
We refill the excised neighborhoods of $\ell_0$ and $\ell_1$ to get a homeomorphism to surgery on $\ell_0 \cup \ell_1$.
The meridianal $1/0$ framing on $\ell_0$ is wrapped $k$ times longitudinally by $f_k$ to the $1/k$ framing on $\ell_0$ (relative to the annular framing).
The meridianal framing for $\ell_1$ is taken to the $-1/k$ framing by $f_k$ because $\ell_1$ is oriented oppositely to $\ell_0$.
Then $f_k$ induces a homeomorphism $f_k:M \cong M_{1/k,-1/k}(\ell_0,\ell_1)$.
\end{proof}
Observe that if we have a torus $S$ in a three manifold $M$, we can split the torus $S$ into two annuli $A$ and $A'$.
This decomposes the torus twist on $S$ into annulus twists on $A$ and $A'$.
Osoinach similarly used a complementary pair of annulus twists to construct an infinite family of knots with homeomorphic Dehn surgeries \cite{ann_twist}.
These family of knots are defined from an annulus presentation.
\begin{mydef}
An annulus presentation $(A,\gamma)$ of a knot $J_0$ in $S^3$ is a diagram exhibiting $J_0$ as framed push offs $\ell_0' \cup \ell_1'$ of the boundary of an annulus $A$ embedded in $S^3$ banded together by a band $\gamma$.
\end{mydef}
The framed push offs $\ell_0' \cup \ell_1'$ are obtained from extending $A$ outwards in the $I$ direction.
They are oriented and the band $\gamma$ must respect these orientations.
In Figure \ref{fig:ann_J_0}, we have drawn a family of annulus presentations index by an integer $n \in \Z$.
The integer $n$ indicates the number of twists we put into the annulus $A$.
These are annulus presentations of knots $J_{-1}[n]$ from \cite{zero_surg_exotic}, we just call this knot $J_0$ suppressing $n$ from the notation.
\begin{prop}
Associated to an annulus presentation $(A,\gamma)$ of a knot $J_0$, there is an infinite family of knots $\{J_k\}_{k \in \Z}$ with zero surgery homeomorphisms $\anntwifam$.
\label{prop:ann_twist}
\end{prop}
\begin{proof}
\begin{figure}
\centering

\subfloat[]{{
    \fontsize{10pt}{12pt}\selectfont
    \def\svgwidth{2in}
    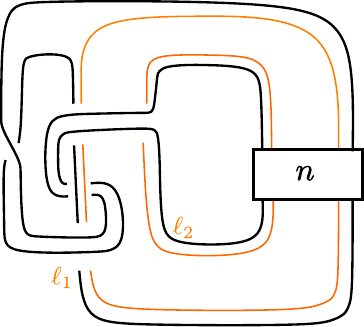
}
\label{fig:ann_J_0}
} \ \
\subfloat[]{{
    \fontsize{10pt}{12pt}\selectfont
    \def\svgwidth{2.2in}
    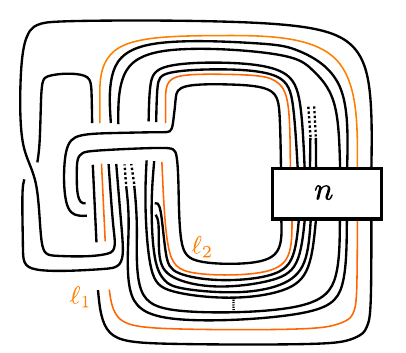
}
\label{fig:ann_J_k}
}

\caption{Annulus presentation of $J_0$ and its $k$ fold annulus twist $J_k$}
\label{fig:ann_pres}
\end{figure}

Since $J_0$ is formed by banding push offs of $\ell_0 \cup \ell_1$, the three cobound a punctured annulus $A_0'$.
Note that $A_0'$ induces the zero framing on $J_0$ since $J_0$ is identified in homology with cancelling pairs of meridians by $A_0' \cup A$.
Therefore, the punctured annulus $A_0'$ can then be capped off to an annulus $A'$ in $S^3_0(J_0)$.
Annulus twisting $A'$ in $S^3_0(J_0)$ shows there is a homeomorphism $S^3_0(J_0) \cong S^3_{0,1/k,-1/k}(J_0,\ell_0,\ell_1)$.
Now a twist on $A$ in $S^3$ gives a homeomorphism $S^3_{1/k,-1/k}(\ell_0,\ell_1) \cong S^3$ and induces a homeomorphism $S^3_{0,1/k,-1/k}(J_0,\ell_0,\ell_1) \cong S^3_0(J_k)$ for some knot $J_k$.
Composing these homeomorphisms then gives the desired $\anntwifam$.
\end{proof}
\begin{remark}
We will often abuse notation and will repeatedly use $\{J_k\}_{k \in \Z}$ to denote several families of different annulus twist knots.
When the notation $\{J_k\}_{k \in \Z}$ is used, it should be thought of as the current family of annulus twists in question.
This should not cause too much confusion as in each section, this notation is never used for two different families of knots.
\end{remark}

Now all of the constructions we have discussed come together for the main result of this section.
We reinterpret Osoinach's annulus twisting construction $4$-dimensionally as a torus surgery.
\begin{thm}
The annulus twist homeomorphism $\anntwifam$ given by an annulus presentation of $J_0$ is the same as a torus surgery of the knot traces.
To be precise, for an annulus presentation $(A,\gamma)$ of $J_0$ there is an associated torus $T \subset X_0(J_0)$ with torus surgery $X_0(J_0)_{T,\tau^k} = (X_0(J_0) -\nu(T)) \cup_{\tau^k} T\times D^2$ such that the annulus twist homeomorphism $\anntwifam$ extends to a diffeomorphism $\Phi_k: X_0(J_0)_{T,\tau^k} \cong X_0(J_k)$.

\label{thm:ann_twi_log_trans}
\end{thm}
A few things to clarify before the proof.
The diffeomorphism $\Phi_k$ extends the annulus twist homeomorphism $\anntwifam$ on the boundary.
For this extension statement to make sense, the torus surgery $X_0(J_0)_{T,\tau^k}$ must have boundary identified with $S^3_0(J_0)$.
The torus surgery is in the interior of $X_0(J_0)$ and leaves the boundary unaffected.
Therefore, the boundaries are identified $\partial X_0(J_0)_{T,\tau^k} = \partial X_0(J_0)$ and the latter is canonically identified with $S^3_0(J_0)$.
The torus $T \subset X_0(J_0)$ is formed from the same $A'$ and $A$ in the proof Proposition \ref{prop:ann_twist} pushed into the interior of $X_0(J_0)$.
The homeomorphism $\tau$ is a torus twist on $S = T \times \{\theta\}$ for some $\theta \in \partial D^2$ in the direction inherited from $A$. 
\begin{figure}
\centering
\subfloat[]{{
    \fontsize{10pt}{12pt}\selectfont
    \def\svgwidth{2in}
    \input{images/ann_twist.pdf_tex}
}
\label{fig:ann_log_0}
} \ \
\subfloat[]{{
    \fontsize{10pt}{12pt}\selectfont
    \def\svgwidth{2in}
    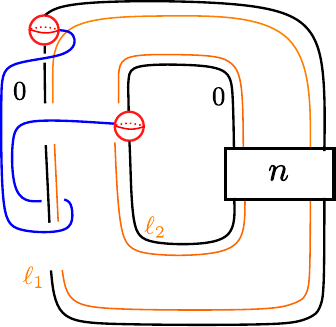
}
\label{fig:ann_log_1}
}

\subfloat[]{{
    \fontsize{10pt}{12pt}\selectfont
    \def\svgwidth{2in}
    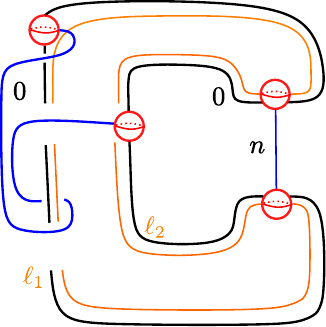
}
\label{fig:ann_log_2}
} \ \
\subfloat[]{{
    \fontsize{10pt}{12pt}\selectfont
    \def\svgwidth{2in}
    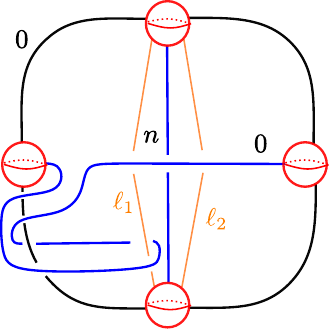
}
\label{fig:ann_log_3}
}
\caption{Annulus twisting as a torus surgery}
\label{fig:ann_log}
\end{figure}
\begin{proof}
We draw diagrams explaining this for the knot $J_0$ shown in Figure \ref{fig:ann_log_0}.
We first modify the diagram of $X_0(J_0)$ so that we clearly see a torus in $X_0(J_0)$ as well as the annuli $A$ and $A'$.
Add a cancelling one and two handle following the band $\gamma$ of $J_0$ as in Figure \ref{fig:ann_log_1}.
The $1$-handles are drawn at the base of the band $\gamma$ on $\partial A$ and the $2$-handle attached to the center of the band with framing induced by the band.
We have to flip the band over near where it meets the annulus in the top left corner since the attaching spheres of the new $1$-handle are identified by an orientation reversing homeomorphism.
The annulus $A'$ now runs over the new one handle and the old two handle.
A good way to see this is to push the balls together along the new $2$-handle until they are almost touching and about to cancel.
Now add cancelling one and two handles along $A$ to get Figure \ref{fig:ann_log_2}.
The framing of the new $2$-handle is the framing induced by the annulus $A$ and by an isotopy, we get Figure \ref{fig:ann_log_3}.
In this picture, the annulus $A$ now runs over the new $1$-handle.

We now see a trivial torus bundle $T \times D^2$ formed by the red $1$-handles and black $2$-handle in Figure \ref{fig:ann_log_3}.
There are then a pair of blue $2$-handles cancelling the red $2$-handles to get our trace $X_0(J_0)$.
Denote the blue $2$-handle running vertically by $B_v$ and the blue $2$-handle running horizontally by $B_h$.
This identifies $S^3_0(J_0) = \partial X_0(J_0)$ as $T \times S^1 = \partial T \times D^2$ with Dehn surgery on the attaching circles of $B_h$ and $B_v$.
In this description of $S^3_0(J_0)$, the homeomorphism $\phi_k$ is induced by the annulus twists on $A$ and $A'$ on $T \times S^1$.
However, the twists on $A$ and $A'$ now match up to become a torus twist $\tau^k$ on $S = T \times \{\theta\}$ for some $\theta \in \partial D^2$.
The torus $S$ runs over the $1$-handles and black $2$-handles and is seen laying flat on the page in Figure \ref{fig:ann_log_3} within the ``outer boundary'' formed by those handles.
Observe that the torus twist and the annulus twist acts on an arc piercing the torus $S$ in the same way.
So $\phi_k$ is now induced by the torus twist $\tau^k$ and carrying along the Dehn surgeries on $B_h,B_v$.
Now the torus surgery on $T$ in $X_0(J_0)$ with $\tau^k$ then has the same effect on the boundary.
We described $X_0(J_0)$ as $T \times D^2$ with $2$-handles attached to $B_h,B_v$.
The support of the torus twist is disjoint from $B_v$, but inherits intersections with $B_h$ from the intersection of the band with the annulus in Figure \ref{fig:ann_log_0}.
The torus twist wraps $B_h$ vertically around the torus at each of these intersections.
Reversing the Kirby Calculus in Figure \ref{fig:ann_log} and cancelling the $2$-handles with the red $1$-handles then matches $X_0(J_k)$ as shown in Figure \ref{fig:ann_J_k}.
\end{proof}
Now suppose we have a zero trace $X_0(J_0)$ embedded in a $4$-manifold $X$ and we replace it with the trace $X_0(J_k)$ of its annulus twist $J_k$ using $\anntwifam$.
The above shows that this is equivalent to a torus surgery of $X$ on a torus $T$ contained within $X_0(J_0)$.
This will allow us to apply the highly developed theory of torus surgery to annulus twisting with knot traces.
\section{Manolescu-Piccirillo Homotopy Spheres}
\label{sec:MP_spheres}
Manolescu and Piccirillo constructed an infinite family of homotopy $4$-spheres $Z_k$ from annulus twist homeomorphisms on a ribbon knot.
These were shown to be standard by the author \cite{trace_emb_zero}.
It was noted that there were similarities to this to work of Akbulut and Gompf on the Cappell-Shaneson spheres \cite{CS_standard,More_Cappell}.
The connection to Akbulut's work was fairly clear, but the connection to Gompf's standardization of the Cappell-Shaneson spheres was more opaque.
Gompf related the various Cappell-Shaneson spheres via torus surgeries.
These did not affect the topology of these homotopy spheres due to the presence of fishtail symmetries.
This prompted the question if the Manolescu-Piccirillo homotopy $4$-spheres could be standardized in a similar manner.
We now answer this question in the affirmative using the results of the previous section.

We first review the construction of the Manolescu-Piccirillo homotopy $4$-spheres $Z_k$.
In the following section we consider the generalization of this construction to elliptic surfaces.
So we will take this as an opportunity to explain this construction in the general case.
Let $X$ be a smooth, closed, oriented, simply connected $4$-manifold for which we want to construct an exotic copy.
We first need an appropriate notion of sliceness in this setting.
Let $X^\circ = X - \inter(B^4)$.
\begin{mydef}
A smoothly, properly embedded disk $D\subset X^\circ$ is an $H$-slice disk in $X$ if $[D] = 0 \in H_2(X^\circ, \partial X^\circ)$.
A knot $K$ is said to be $H$-slice in $X$ if $K$ is the boundary $K = \partial D$ of some $H$-slice disk $D$ in $X$.
\end{mydef}
This generalizes classical sliceness: a knot is slice in $B^4$ if and only if it is $H$-slice in $S^4$.
If $K$ bounds an $H$-slice disk $D$ in $X$, let $\nu(D)$ be a tubular neighborhood of $D$ and let the exterior of $D$ be $E(D) = X^\circ - \nu(D)$.
This exterior $E(D)$ has boundary naturally identified with $S^3_0(K)$.
We can fill the disk back in by attaching a $2$-handle and capping off with a $4$-handle to recover $X$.
These additional handles are a nullhomologously embedded trace $-X_0(K)$ decomposing $X$ as $E(D) \cup -X_0(K)$.
$H$-slice trace surgery is then the process of removing $-X_0(K)$ and replacing it with a knot trace $-X_0(K')$ using some zero surgery homeomorphism.
\begin{mydef}
Let $D$ be an $H$-slice disk in $X$ with $\partial D = K$ and let $\zsg$ be a zero surgery homeomorphism.
The $4$-manifold $X_{D,\phi} = E(D) \cup_\phi -X_0(K')$ is called the $H$-slice trace surgery of $X$ with $D$ and $\phi$.
\end{mydef}
%
The $H$-slice trace surgery $X_{D,\phi}$ is also simply connected since $X$ is (Lemma $3.3$ of \cite{zero_surg_exotic}) and therefore is homeomorphic to $X$ by Freedman \cite{Freedman}.

In \cite{rel_gen_indef_manifold}, it was shown that $H$-sliceness can be used to distinguish exotic pairs of $4$-manifolds.
For example, $T_{2,3}$ is not $H$-slice in $K3 \# \CPbar$, but is $H$-slice in $3 \CP \# 20 \CPbar$.
Such a knot is called exotically $H$-slice with respect to these two $4$-manifolds.
This occurs when an $H$-slice trace surgery $X_{D,\phi}$ has $K'$ not $H$-slice in $X$.
Then $K'$ is exotically $H$-slice in $X_{D,\phi}$ with respect to $X$ and $X_{D,\phi}$ is an exotic $X$.
In this case we say that $X_{D,\phi}$ is a \textit{strongly exotic H-slice trace surgery}.
However, if $K'$ is $H$-slice in $X$, it does not necessarily mean that $X_{D,\phi}$ is diffeomorphic to $X$.
When $K'$ is $H$-slice in $X$, but $X_{D,\phi}$ is differentiated from $X$ by other means, we say that $X_{D,\phi}$ is a \textit{weakly exotic H-slice trace surgery}.
This is a more interesting distinction than it may appear at first as we will see in Section \ref{sec:questions}

For $H$-slice trace surgeries on $X = S^4$, we simply call this a slice trace surgery.
Manolescu and Piccirillo hoped to find a strongly exotic slice trace surgery as a counterexample to SPC4.
They constructed an infinite family of homotopy $4$-spheres $Z_k$ via slice trace surgery as candidates for a strongly exotic slice trace surgery.
These $Z_k = S^4_{D,\phi_k}$ were formed by annulus twisting $\anntwifam$ a ribbon disk $D$ with boundary $\partial D = J_0$.
%
%
The ribbon disk $D$ and annulus twist homeomorphisms $\phi_k$ are specified by Figure \ref{fig:ann_torus_1}.
Note that this is the same annulus presentation as Figure \ref{fig:ann_log} with $n = 2$.

The resulting slice trace surgeries $Z_k = S^4_{D,\phi_k}$ were exciting candidates for counterexamples to SPC4.
If any of the knots $J_k$ were not slice, then $Z_k$ would be a strongly exotic slice trace surgery and so an exotic $4$-sphere.
Unfortunately, the author showed that these $Z_k$ were standard and that each $J_k$ was slice.
This was accomplished by drawing Kirby diagrams that depicted the trace embeddings $X_0(J_k) \subset -Z_k$ \cite{trace_emb_zero}.
We now use the techniques of the previous section to give a mostly Kirby diagram free proof that the homotopy $4$-spheres $Z_k$ are standard.
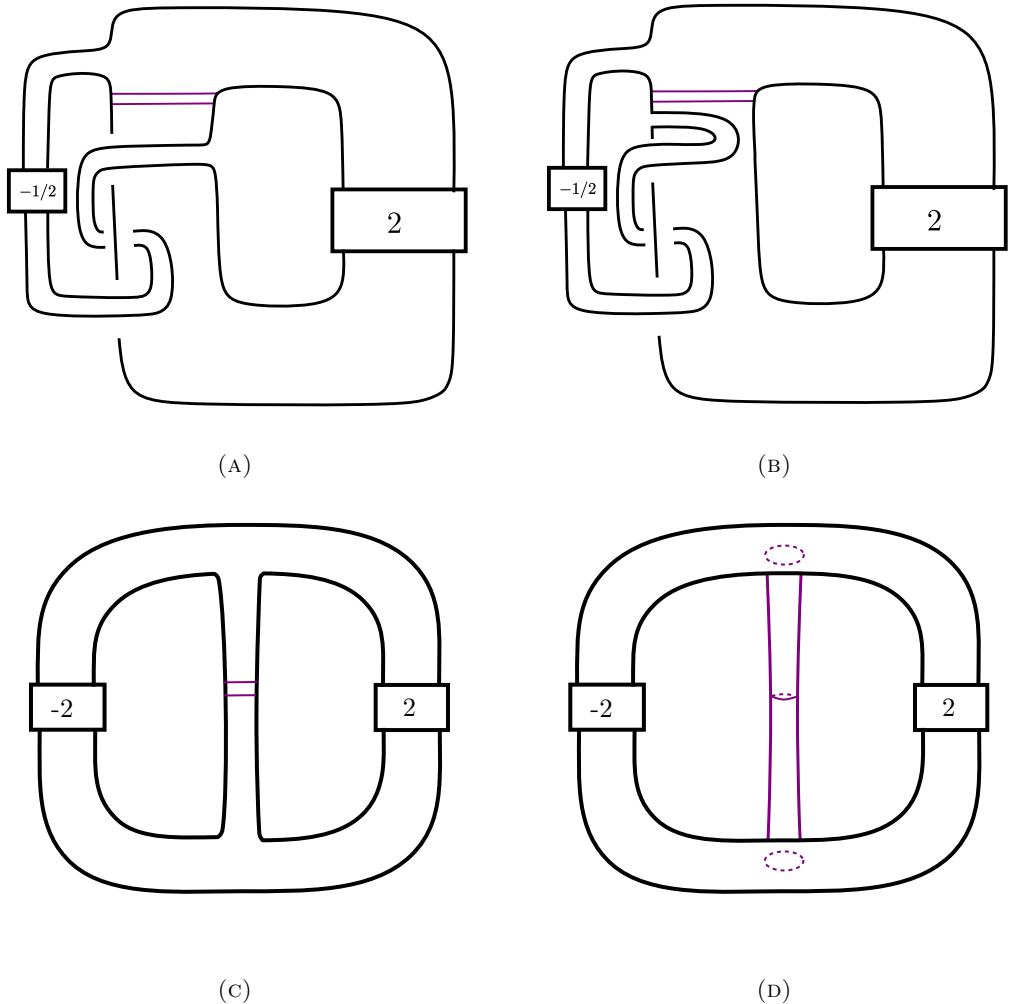
\begin{figure}
\centering
\subfloat[]{{
    \fontsize{10pt}{12pt}\selectfont
    \def\svgwidth{2.5in}
    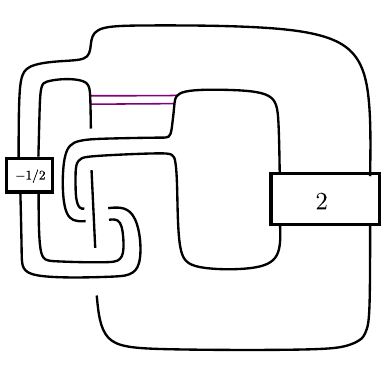
}
\label{fig:ann_torus_1}
} \ \
\subfloat[]{{
    \fontsize{10pt}{12pt}\selectfont
    \def\svgwidth{2.5in}
    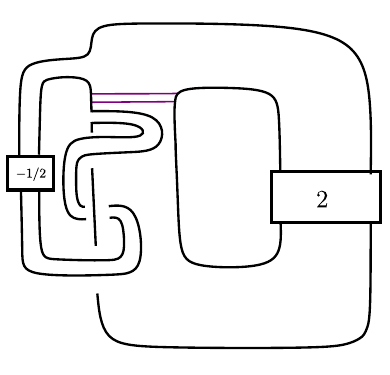
}
\label{fig:ann_torus_2}
} \ \
\subfloat[]{{
    \fontsize{10pt}{12pt}\selectfont
    \def\svgwidth{2.5in}
\begingroup%
  \makeatletter%
  \providecommand\color[2][]{%
    \errmessage{(Inkscape) Color is used for the text in Inkscape, but the package 'color.sty' is not loaded}%
    \renewcommand\color[2][]{}%
  }%
  \providecommand\transparent[1]{%
    \errmessage{(Inkscape) Transparency is used (non-zero) for the text in Inkscape, but the package 'transparent.sty' is not loaded}%
    \renewcommand\transparent[1]{}%
  }%
  \providecommand\rotatebox[2]{#2}%
  \newcommand*\fsize{\dimexpr\f@size pt\relax}%
  \newcommand*\lineheight[1]{\fontsize{\fsize}{#1\fsize}\selectfont}%
  \ifx\svgwidth\undefined%
    \setlength{\unitlength}{180bp}%
    \ifx\svgscale\undefined%
      \relax%
    \else%
      \setlength{\unitlength}{\unitlength * \real{\svgscale}}%
    \fi%
  \else%
    \setlength{\unitlength}{\svgwidth}%
  \fi%
  \global\let\svgwidth\undefined%
  \global\let\svgscale\undefined%
  \makeatother%
  \begin{picture}(1,1)%
    \lineheight{1}%
    \setlength\tabcolsep{0pt}%
    \put(0,0){\includegraphics[width=\unitlength,page=1]{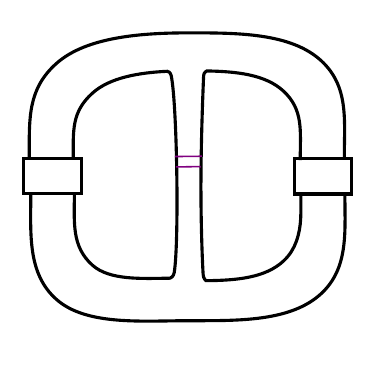}}%
    \put(0.10626833,0.51037272){\color[rgb]{0,0,0}\makebox(0,0)[lt]{\lineheight{1.25}\smash{\begin{tabular}[t]{l}-2\end{tabular}}}}%
    \put(0.84254984,0.51102865){\color[rgb]{0,0,0}\makebox(0,0)[lt]{\lineheight{1.25}\smash{\begin{tabular}[t]{l}2\end{tabular}}}}%
  \end{picture}%
\endgroup%

}
\label{fig:ann_torus_3}
} \ \
\subfloat[]{{
    \fontsize{10pt}{12pt}\selectfont
    \def\svgwidth{2.5in}
\begingroup%
  \makeatletter%
  \providecommand\color[2][]{%
    \errmessage{(Inkscape) Color is used for the text in Inkscape, but the package 'color.sty' is not loaded}%
    \renewcommand\color[2][]{}%
  }%
  \providecommand\transparent[1]{%
    \errmessage{(Inkscape) Transparency is used (non-zero) for the text in Inkscape, but the package 'transparent.sty' is not loaded}%
    \renewcommand\transparent[1]{}%
  }%
  \providecommand\rotatebox[2]{#2}%
  \newcommand*\fsize{\dimexpr\f@size pt\relax}%
  \newcommand*\lineheight[1]{\fontsize{\fsize}{#1\fsize}\selectfont}%
  \ifx\svgwidth\undefined%
    \setlength{\unitlength}{180bp}%
    \ifx\svgscale\undefined%
      \relax%
    \else%
      \setlength{\unitlength}{\unitlength * \real{\svgscale}}%
    \fi%
  \else%
    \setlength{\unitlength}{\svgwidth}%
  \fi%
  \global\let\svgwidth\undefined%
  \global\let\svgscale\undefined%
  \makeatother%
  \begin{picture}(1,1)%
    \lineheight{1}%
    \setlength\tabcolsep{0pt}%
    \put(0,0){\includegraphics[width=\unitlength,page=1]{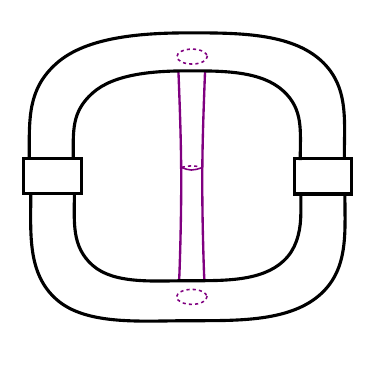}}%
    \put(0.10626833,0.51037272){\color[rgb]{0,0,0}\makebox(0,0)[lt]{\lineheight{1.25}\smash{\begin{tabular}[t]{l}-2\end{tabular}}}}%
    \put(0.84254984,0.51102865){\color[rgb]{0,0,0}\makebox(0,0)[lt]{\lineheight{1.25}\smash{\begin{tabular}[t]{l}2\end{tabular}}}}%
  \end{picture}%
\endgroup%

}
\label{fig:ann_torus_4}
}
\caption{Realizing slice trace surgery as an unknotted torus surgery}
\label{fig:ann_torus}
\end{figure}

\begin{thm}
The Manolescu-Piccirillo homotopy $4$-spheres $Z_k$ are standard.
\end{thm}
\begin{proof}
Theorem \ref{thm:ann_twi_log_trans} tells us that we can view $Z_k$ as a torus surgery on the $4$-sphere.
We will find it easier however to turn everything upside down first and consider $-Z_k$.
The torus surgery to $-Z_k$ is contained within the trace $X_0(J_0) \subset -Z_0 = S^4$.
The torus $T$ in the trace $X_0(J_0)$ is formed by two annuli $A$ and $A'$ pushed into the interior of $X_0(J_0)$.
Fortunately, Figure \ref{fig:ann_torus_1} also describes this torus $T$ as a level set picture.

The first annulus $A$ is seen immediately in Figure \ref{fig:ann_torus_1}.
Recall that the second annulus is formed from the band and the surgery disk.
However, the surgery disk is the core of the $2$-handle of $X_0(J_0)$ which in the embedding $X_0(J_0) \subset S^4$ is given by the disk $D$.
We use the ribbon move in Figure \ref{fig:ann_torus_1} and cap off with two disks to draw the torus $T$ as in Figure \ref{fig:ann_torus_1}.
That figure and the following two figures omits the two capping off disks which we will need to remember to carry along throughout these manipulations.

We can now simplify, slide the $-1/2$ twisted band over the purple band to get Figure \ref{fig:ann_torus_2}.
Now we can unwrap the $-1/2$ twisted band from it's linking with the rest of the diagram so it is a $-2$ twisted band symmetric with the other band to get Figure \ref{fig:ann_torus_3}.
The purple band in Figure \ref{fig:ann_torus_3} can now be expanded into a purple tube connecting the annulus together as in Figure \ref{fig:ann_torus_4}.
Now untwist the annulus in Figure \ref{fig:ann_torus_4} to reduce the twist boxes to $\pm 1$ and wrapping the purple tube around the annulus once.
We can unwrap the purple tube from the annulus by passing down to a lower level disjoint from the annulus, bring it back to its original position in that level, and then bring it back up.
This recovers Figure \ref{fig:ann_torus_4} with $\pm 1$ twists in place of the $\pm 2$ twists, repeat this process again to fully remove the twists.
Capping off with two disks, we clearly see that this is the unknotted torus and so any simply connected torus surgery on it is standard \cite{larson_tori}.
\end{proof}
%
\section{Exotically annulus twisting}
\label{sec:constr}
One of the most effective constructions of exotic $4$-manifolds is Fintushel-Stern knot surgery \cite{Fint_Ster_knot_surg}.
For an elliptic surface $X = E(n)$ and a knot $K$, this construction produces a $4$-manifold $X_K$ homeomorphic to $X$ which is often not diffeomorphic to $X$.
Each $E(n)$ admits a singular fibration over the sphere with generic fiber a torus and $12n$ critical points.
This description allows us to build $E(n)$ up starting with the fiber $\torusnbhd$ over a generic disk.
Passing over a critical point contributes a $2$-handle attached to $\torusnbhd$ along a copy of alternating $S^1$ factors with framing $-1$ relative to the product framing.
Observe that the unknot $U$ has exterior $E(U) = S^1 \times D^2$ and so $\torusnbhd$ can then be thought of as $S^1 \times E(U)$.
This identifies a longitude of $U$ with $\partial D^2$ and a meridian of $U$ with the middle $S^1$ factor of $\torusnbhd$.
Then $X_K$ is obtained by removing $\torusnbhd = S^1 \times E(U)$ from $X$ and replacing it with $S^1 \times E(K)$ matching longitudes and meridians in the same way.

The real power of this construction is realized by Fintushel and Sterns' calculation of the Seiberg-Witten invariants of $X_K$ in terms of the Alexander polynomial $\Delta_K(t)$ of $K$.
They express the knot surgery as a sequence of torus surgeries on nullhomologous tori.
The effect of these torus surgeries on the Seiberg-Witten invariants satisfy the same sort of skein relation that the Alexander polynomial satisfies.
In particular, if $\Delta_{K}(t)$ and $\Delta_{K'}(t)$ are different, then $X_{K}$ and $X_{K'}$ are not diffeomorphic \cite{knot_surg_note}.
For $k \in \Z$, let $X_k$ be Fintushel-Stern knot surgery on $X$ with the $k$-twist knot $\kappa_k$ shown in Figure \ref{fig:twist_knot_1}.
Note that $\kappa_0$ is the unknot and therefore $X_0 = X$.
The twist knots $\kappa_k$ have distinct Alexander polynomials and therefore $X_k$ are an infinite exotic family of $4$-manifolds.
Gompf constructed an infinite order cork by starting with a nullhomologous torus $T$ in $X_0$ such that a torus surgery on $T$ gives $X_k$.
The infinite order cork is obtained by absorbing $T \times D^2$ into a contractible submanifold in $X_0$ so that the the torus twist survives into the boundary \cite{inf_ord_cor}.
Our construction of exotic $H$-slice trace surgeries is directly inspired by Gompf's infinite order cork.
We start with the same $X_k$ that are already established as exotic.
Then we use the same torus surgery along with Theorem \ref{thm:ann_twi_log_trans} to realize these $X_k$ as $H$-slice trace surgeries.
\begin{figure}
\centering

\subfloat[]{{
    \fontsize{10pt}{12pt}\selectfont
    \def\svgwidth{1.5in}
    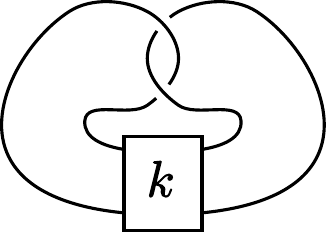
}
\label{fig:twist_knot_1}
} \ \
\subfloat[]{{
    \fontsize{10pt}{12pt}\selectfont
    \def\svgwidth{1.5in}
    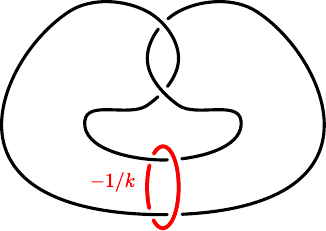
}
\label{fig:twist_knot_2}
}
\caption{(a)Twist knot $\kappa_k$ (b) Surgery description of $\kappa_k$ as $\kappa_0$ and $-1/k$ surgery on the red knot $C$ }
\label{fig:twist_knot}
\end{figure}

\begin{thm}
There is an $H$-slice disk $D$ in $X = E(n)$ with boundary $K_0$ and annulus twist homeomorphisms $\phi_k:S^3_0(K_0) \rightarrow S^3_0(K_k)$ such that each $X_k$ can be realized as $H$-slice trace surgery $X_{D,\phi_k}$ on $X$ with $D$ and $\phi_k$.
\label{ex_Hslice_full}
\end{thm}
\begin{proof}
To prove the theorem it will suffice to decompose $X_k$ as $Y \cup_{\phi_k \circ f} X_0(J_k)$ with $X_0(J_0)$ nullhomologously embedded.
Here $f$ is some boundary identification $f:\partial Y \cong S^3_0(J_0)$ and $\phi_k$ are annulus twist homeomorphisms $\phi_k:S^3_0(J_0) \rightarrow S^3_0(J_k)$.
We take $K_k$ in the theorem to be $-J_k$ so that $-X_0(K_k) = X_0(J_k)$ and $X_k = Y \cup_{\phi_k \circ f} -X_0(K_k)$.
The nullhomologous trace embedding of $-X_0(K_0)$ into $X_0$ exhibits an $H$-slice $D$ in $X_0$ for $K_0$ by the $H$-slice trace embedding lemma (Lemma $3.5$ of \cite{zero_surg_exotic}).
Then $X_k = X_{D,\phi_k}$ establishing the theorem.

Each $\kappa_k$ has a surgery description as $\kappa_0$ with $-1/k$ Dehn surgery on the red knot $C$ in Figure \ref{fig:twist_knot_2}.
The required Dehn surgery is performed by removing a solid torus neighborhood $V$ of $C$ and regluing by $k$-fold Dehn twist $\tau^k$ along a longitude $L \subset \partial V$.
Then $S^1 \times E(K_k)$ is obtained from $S^1 \times E(K_0)$ by cutting out $S^1 \times V$ and regluing it by $Id \times \tau^k$.
This expresses $X_k$ as a torus surgery on the torus $T = S^1 \times C$ contained within $S^1 \times E(\kappa_0) \subset X_0$ and gluing map a torus twist on $S^1 \times L$. 
The torus $T = S^1 \times C$ is nullhomologous in $X_0$ because $C$ is nullhomologous in $E(\kappa_0)$.
We wish to ambiently attach $2$-handles to $S^1 \times V$ within $X_0$ so that we get an embedded knot trace.
The handle product decomposition of $S^1 \times V$ has a pair of $1$-handles coming from each each of the $S^1$ and $V$ factors.
The Fintushel-Stern construction automatically matches the $1$-handle coming from the $S^1$ factor to a cancelling $2$-handle.
For the $V$ factor, observe there is a twice punctured disk $\Sigma_0$ in $E(\kappa_0 \cup c)$ with boundary a longitude of $V$ and two meridians of $\kappa_0$.
The Fintushel-Stern construction then matches the meridians of a $\{\theta\} \times \Sigma_0$ with the cores of a pair of $2$-handles.
This results in a disk $\Sigma$ with boundary running once over a copy of $V$ in $S^1 \times V$ which we thicken to get a $2$-handle.

\begin{figure}
\centering

\subfloat[]{{
    \fontsize{10pt}{12pt}\selectfont
    \def\svgwidth{2in}
    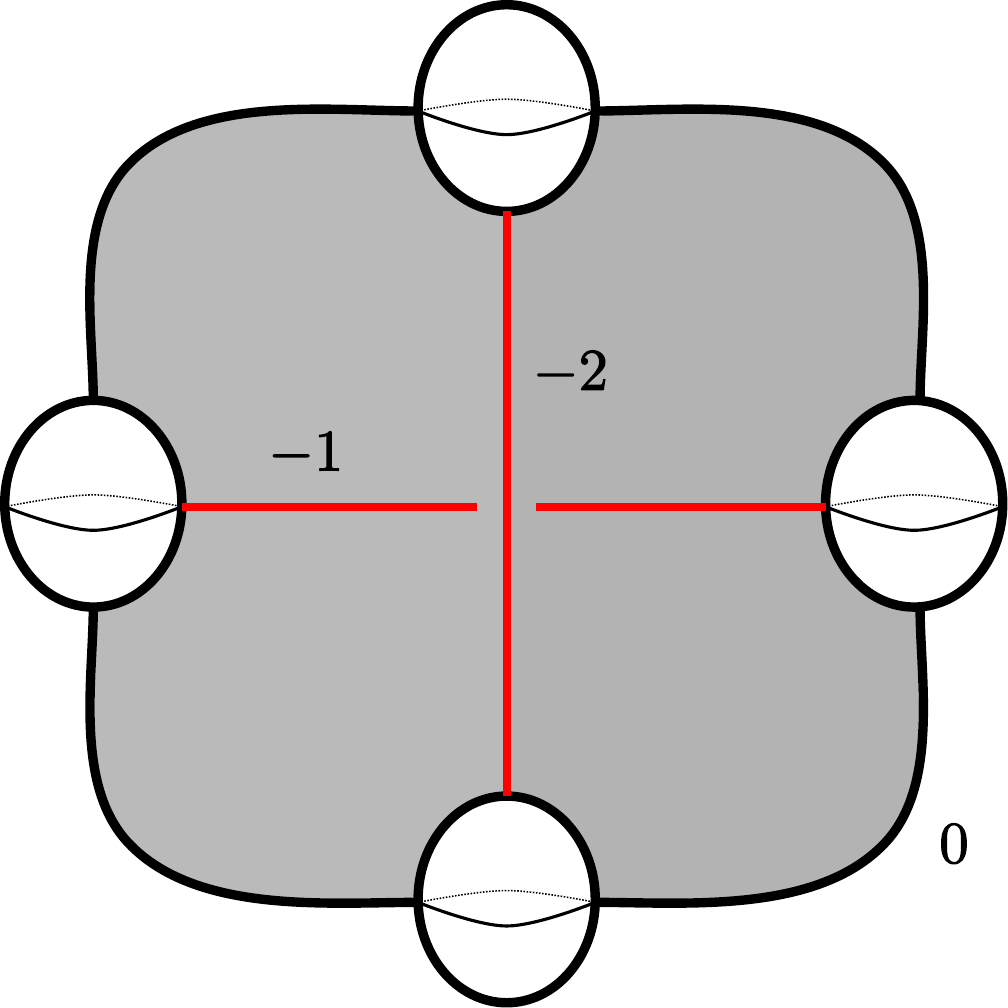
}
\label{fig:kirby_trace}
} \ \
\subfloat[]{{
    \fontsize{10pt}{12pt}\selectfont
    \def\svgwidth{2in}
    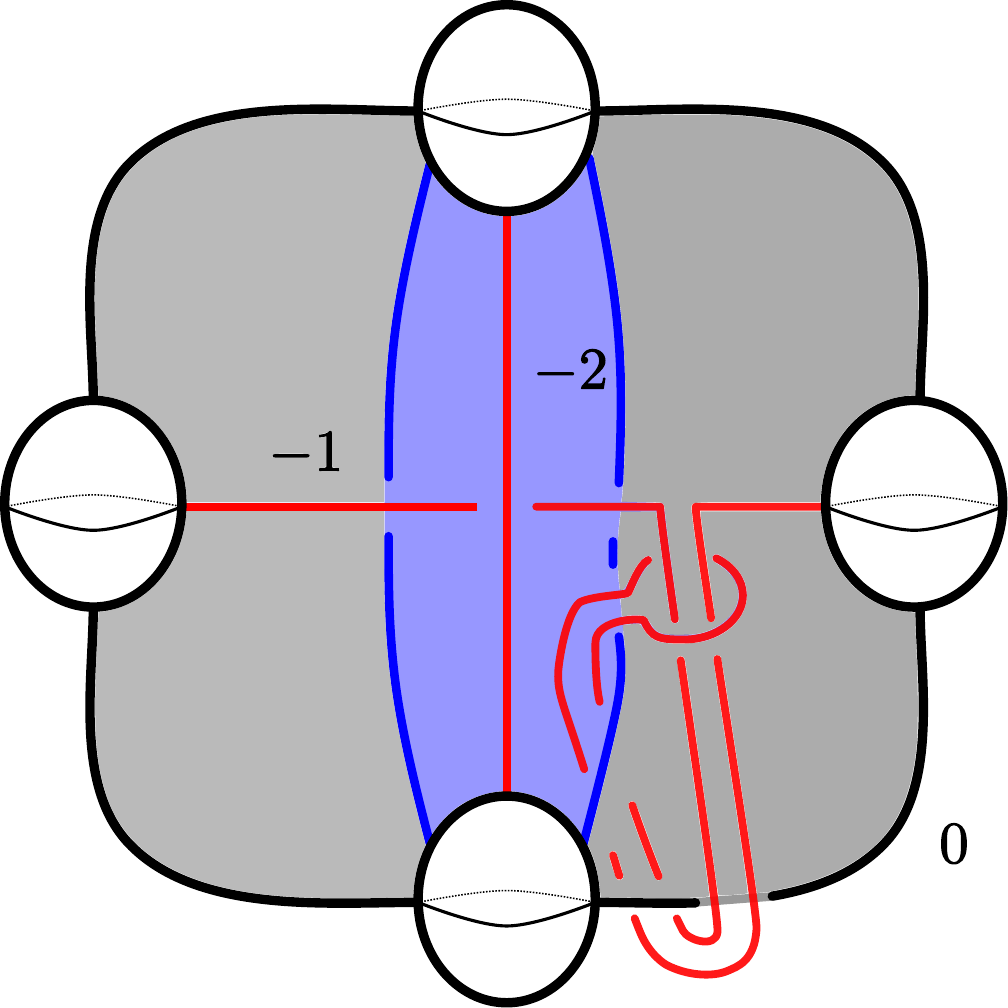
}
\label{fig:kirby_trace_ribbon}
}
\caption{Kirby diagrams of our knot traces}
\label{fig:kirby_diagram}
\end{figure}
In Figure \ref{fig:kirby_trace}, we have drawn a Kirby diagram of $S^1 \times V$ with the $2$-handles attached.
To get this figure, we start with a diagram of $S^1 \times V$ given by ignoring the red $2$-handles and the grey shaded region in Figure \ref{fig:kirby_trace}.
Each $2$-handle core we used to cap off to get disks decreases the framing of the corresponding thickened $2$-handle by $-1$ from the product framing.
The new $2$-handles are unaffected by the torus twist on $S^1 \times L$ and so the $4$-manifold is left unchanged.
Cancelling the handles in Figure \ref{fig:kirby_trace}, we see that $X_k$ is obtained from $X_0$ by cutting out $X_0(\kappa_{-2})$ and gluing it back in by a twist.
This is more akin to a cork and to get something more interesting, we modify the $-1$ framed $2$-handle in Figure \ref{fig:kirby_trace} by a ribbon concordance.
Observe that there is a concordance from the $-1$ framed $2$-handles going from Figure \ref{fig:kirby_trace_ribbon} to \ref{fig:kirby_trace}.
This concordance fits together with the core of the $-1$ framed $2$-handle in Figure \ref{fig:kirby_trace} to make up the core of the $2$-handle in Figure \ref{fig:kirby_trace_ribbon}.
The torus twist on $S^1 \times L$ now wraps the modified $2$-handle over the vertical $1$-handle.
The $1$-handles still cancel after twisting and we now get a family of knots with zero surgery homeomorphisms $\anntwifam$.
The trace $X_0(J_0)$ is nullhomologously embedded in $X_0$ because the torus $S^1 \times C$ still represents a generator of $H_2(X_0(J_0))$ and is nullhomologous in $X_0$.

The zero surgery homeomorphisms $\anntwifam$ induced by the torus surgeries are of the same kind constructed in the proof of Theorem \ref{thm:ann_twi_log_trans}.
Therefore, $\phi_k$ are given by annulus twisting $J_0$.
By working through the proof of that theorem backwards cancelling the $1$-handles, we get the corresponding annulus presentation.
This is best done by first swinging the vertical red $2$-handle in Figure \ref{fig:kirby_trace_ribbon} outside and on the left of the figure before cancelling.
Then pull the corners of $J_0$ in so that they run parallel to the boundary of the blue annulus until they band over the horizontal $1$-handle.
Then cancelling the horizontal $1$-handle leaves a band in the black $2$-handle following the red $2$-handle.
Then swing the inner boundary of the annulus behind and under to the front to clean up the diagram getting Figure \ref{fig:ex_ann_twist}.

\begin{figure}
\centering

    \fontsize{10pt}{12pt}\selectfont
    \def\svgwidth{2.5in}
    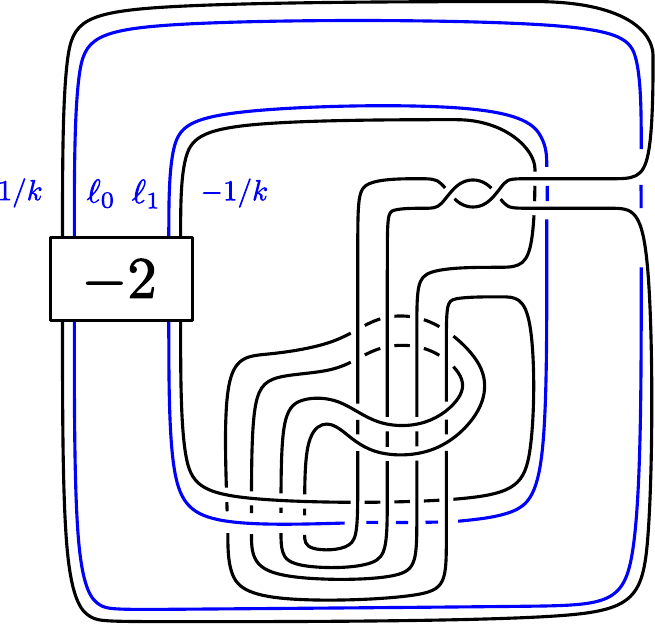
\caption{Annulus presentation of $J_0$ realizing the homeomorphism $\anntwifam$. The framings on $\ell_0$ and $\ell_1$ are relative to the annulus framings.}
\label{fig:ex_ann_twist}
\end{figure}
\end{proof}
Observe that if we cancel the $1$ and $2$-handles in Figure \ref{fig:kirby_trace}, we get the zero trace of $\kappa_{-2}$.
So we have a nullhomologous trace embedding of $-X_0(-\kappa_{-2})$ and can conclude that $-5_2 = -\kappa_{-2}$ is $H$-slice in every elliptic surface $E(n)$.
In particular, the $-5_2$ knot is $H$-slice in the $K3$ surface.
Manolescu, Marengon, and Piccirillo studied sliceness of knots in $K3$ and showed that all knots with unknotting number at most two are slice in $K3$ \cite{rel_gen_indef_manifold}.
Marengon and Mihajlovi{\'c} later extended to unknotting number $21$ \cite{marengon2022unknotting}.
This shows that general sliceness is a much weaker notion for the $K3$ surface than it is for the classical setting in $B^4$.
However, with $H$-sliceness we see a very different picture.
\begin{thm}
The $-5_2$ knot is $H$-slice in the $K3$-surface and bounds an $H$-slice disk with infinitely many distinct $H$-slice trace surgeries.
The $-5_2$ knot is the smallest crossing number non-trivial knot that is $H$-slice in $K3$ and is the only prime non-slice knot with six or fewer crossings that is $H$-slice in $K3$.
\label{thm:Hslice_in_K3}
\end{thm}
\begin{proof}
The knot $\pm 6_1$ is slice in $B^4$ and therefore trivially $H$-slice in $K3$.
So we exclude $\pm 6_1$ from discussion.
$H$-sliceness of $-\kappa_2=-5_2$ in $K3$ follows from the proof of Theorem \ref{ex_Hslice_full}, however it can easily be shown directly.
The $K3$ surface admits a Lefschetz fibration with a $(-2)$-sphere as a section.
This implies that there is a disk $D \subset K3^\circ$ with self intersection number $[D]^2 = -2$ and unknotted boundary $\partial D = U$.
By Lemma $2.10$ of \cite{rel_gen_indef_manifold}, the twisted Whitehead double $WH^+_2(U)$ is $H$-slice in $K3$ which can be easily recognized as $-5_2$.

To show that no other non-slice knot with six or fewer crossings is $H$-slice in $K3$, we first use the arf invariant.
Any knot $H$-slice in a spin $4$-manifold such as $K3$ must have vanishing arf invariant by Robertello \cite{Robertello}.
Knotinfo shows that among prime knots with six or fewer crossings, only $5_2$ and $6_1$ have trivial Arf invariant \cite{knotinfo}.
We are not quite done yet, as we have to show that $5_2$ is not $H$-slice in $K3$.
This is because $H$-sliceness is sensitive to chirality, unlike classical sliceness in $B^4$.
One can easily draw a Legendrian diagram of $5_2$ with Thurston-Bennequin number $1$ and conclude that $5_2$ is not $H$-slice in $K3$ by Corollary $1.8$ of \cite{bauer_furuta_Hslice}.
\end{proof}
\begin{remark}
It turns out that the $H$-slice disks for $-5_2$ in $K3$ constructed in the proofs of Theorems \ref{ex_Hslice_full} and \ref{thm:Hslice_in_K3} are distinct.
This can be seen by explicitly checking that the smooth structure of $K3$ is unchanged by cutting and gluing the embedded trace corresponding to the $H$-slice disk in Theorem \ref{thm:Hslice_in_K3}.
This is most easily done using a Kirby diagram of $X_{-2}(U)$ showing the embedding of $-X_0(-5_2)$.
Moreover, the construction of $X_1$ as an $H$-slice trace surgery on $X_0 = E(n)$ can be mimicked for any unknotting number one knot $K$ since the unknotting operation can be realized as a single nullhomologous twist as in the proof of Theorem \ref{ex_Hslice_full}.
Dunfield and Hoffman give an infinite family of unknotting number one knots $L_k$ such that the knot surgeries $X_{L_k}$ are all distinct \cite{hoff_sun}.
Running the construction of Theorem \ref{ex_Hslice_full} with these $L_k$ then gives an infinite collection of $H$-slice disks $D_k$ for $-5_2$.
These $D_k$ are then distinct as they have non-diffeomorphic $H$-slice trace surgeries $X_{D_k,\phi} = X_{L_k}$.
It may be possible to show these $D_k$ are topologically isotopic by using recent work on topological classification of $H$-slice disks \cite{ex_disks}.
However, we do not pursue that point here.
\end{remark}
The ribbon move going from Figure \ref{fig:kirby_trace} to Figure \ref{fig:kirby_trace_ribbon} was done to change the exotica rel boundary to absolute exotica.
The ribbon destroys the essential torus in the boundary which would obstruct the presence of a hyperbolic boundary.
Once we have a hyperbolic structure, Mostow rigidity would then guarantee the mapping class group agrees with the isometry group and is finite.
Generically this should be trivial.
Then there would be no non-identity boundary automorphisms that could extend to a diffeomorphism.
This is actually an application of a technique due to Akbulut and Ruberman that converts relatively exotic $4$-manifolds into absolutely exotic $4$-manifolds \cite{akb_rub}.
Here we make their construction explicit.
This allows us control over the handle structure and get exotic traces.
\begin{thm}
\label{thm:ex_trace_full}
There exists a family of knots with mutually exotic traces that
\begin{enumerate}
    \item Consists of infinitely many knots
    \item Related by annulus twisting
    \item Have the same slice genus
    \item Concordant to each other
    \item Have the same zero shake genus
    \item The knots' traces have the same genus function
\end{enumerate}
\end{thm}
\begin{proof}
\begin{figure}
\centering

    \fontsize{10pt}{12pt}\selectfont
    \def\svgwidth{3.5in}
    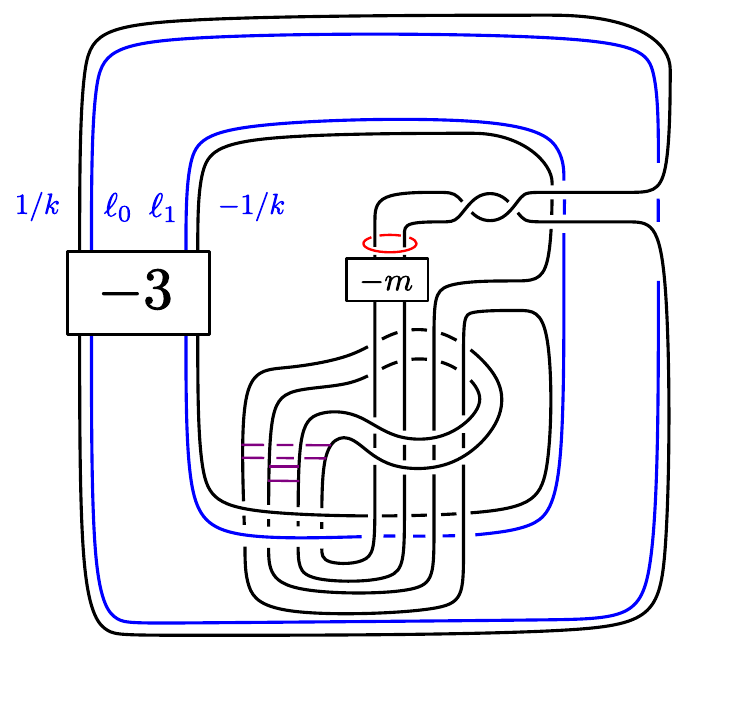
\caption{Annulus presentation of $J_0^m$ giving exotic knot traces. The framings on $\ell_0$ and $\ell_1$ are relative to the annulus framings.}
\label{fig:blow_up_ann_twist}
\end{figure}

First we blow up our $X_k$ by points on cores of the red $2$-handles in Figure \ref{fig:kirby_trace_ribbon}.
Blow up one point to change the $-2$ framing to $-3$ and $m \ge 0$ points to decrease the $-1$ framing to $-1-m$.
The ambient manifolds become the blowups $X_k \# (m+1)\CPbar$ which remain non-diffeomorphic.
Cancelling the $1$-handles we get the family of annulus twists $J^m_k$ in Figure \ref{fig:blow_up_ann_twist}.
The annulus twist homeomorphisms $\phi_k: S^3_0(J^m_n) \cong S^3_0(J^m_{n+k})$ cannot extend to a trace diffeomorphism $\Phi_k: X_0(J^m_n) \cong X_0(J^m_{n+k})$ .
Otherwise the act of cutting and gluing these traces would result in diffeomorphic total spaces.
However, we need to rule out the possibility that another homeomorphism $S^3_0(J^m_n) \cong S^3_0(J^m_{n+k})$ extends.
Observe that $S^3_0(J^m_0)$ can be described as $1/m$ surgery on the red curve $\gamma$ in $S^3_0(J_0^0)$ in Figure \ref{fig:blow_up_ann_twist}.
Using SnapPy, we check that $S^3_0(J^0_0) - \gamma$ is hyperbolic with a unique self homeomorphism up to isotopy \cite{SnapPy}.
Then for sufficiently large $m$, Lemma $2.2$ of \cite{asymm} guarantees that $S^3_0(J^m_0)$ is hyperbolic and asymmetric.
We now check that $\phi_k$ extends to a trace homeomorphism.
This follows from work of Boyer \cite{boyer}, which Manolescu-Piccirillo applied to give a handy criterion for zero traces: a zero surgery homeomorphism $\zsg$ extends to a trace homeomorphism if and only if the twisted double $X_0(K) \cup_\phi -X_0(K')$ has even intersection form (Theorem $3.7$ of \cite{zero_surg_exotic}).
Then the argument in Remark $6.6$ of \cite{zero_surg_exotic} applies in the exact same way to show that $\phi_k$ extends to a homeomorphism.
Now to see that all $J^m_k$ are concordant, observe that there is a pair of ribbon moves in Figure \ref{fig:blow_up_ann_twist} shown in purple.
These ribbon moves separate the band in the annulus diagram so that it becomes disjoint from the annulus resulting in the knot $\kappa(-3,-1-m)$.
This gives a ribbon concordance from $J^m_k$ to a common knot and so all $J_k$ are concordant to each other.
By the last bullet point before Theorem \ref{thm:ex_trace}, this concordance implies that these $J^m_k$ all have the same slice genus, shake genus, and the knots' traces have the same genus function.
\end{proof}
The above proof only gives \textit{existence} of the claimed exotica.
We are not able to explicitly pin down a family of annulus twists that realizes our claims.
However, the generic asymmetry of hyperbolic $3$-manifolds given by Mostow rigidity is very powerful.
We expect that any choice of $m \ge 0$ above should work.
For example, SnapPy calculates the symmetry group of $S^3_0(J^0_0)$ as trivial.
However, this calculation is not verified for closed $3$-manifolds and so we can not use it in a proof.
The calculation should be good evidence that $\{J_k^0\}$ is a family of annulus twists that realizes the claims of Theorem \ref{thm:ex_trace_full}.

After the first counterexamples to the Akbulut-Kirby Conjecture appeared following Yasui,
Miller and Piccirillo noted that all known zero surgery homeomorphisms for which the Akbulut-Kirby Conjecture held were of slice knots.
It seemed that the Akbulut-Kirby Conjecture is false generally unless sliceness gives a reason for it to hold.
So they asked if there were any examples of distinct non-slice knots $K$ and $K'$ such that $S^3_0(K) \cong S^3_0(K')$ and $K$ and $K'$ are smoothly concordant (Question $1.16$ of \cite{trace_conc}).
The zero surgery homeomorphisms of Theorem \ref{thm:ex_trace} answer this question in the positive.
In the following section, we show that we can put a Stein structure on these knots' traces so the adjunction inequality implies that they are not slice.

\section{Exotic Stein Traces}
\label{sec:stein}
The boundary of a Stein surface naturally inherits the geometric structure of a tight contact $3$-manifold.
We call such a Stein surface a Stein filling of the boundary contact $3$-manifold.
There has been much work classifying Stein fillings of a contact $3$-manifold.
For example, Eliashberg showed that the $3$-sphere has a unique Stein filling \cite{fillingS3} and Wendl showed that $3$-torus does as well \cite{filltorus}.
Complementing this, it is of much interest then to construct distinct fillings of a contact $3$-manifold.
Following Eliashberg \cite{stein_yasha}, a natural way to construct Stein surfaces is by Stein handle attachments to Legendrian knots in the contact boundary of a Stein $\natural n S^1 \times B^3$.
When attaching a Stein handle to a Legendrian knot in $B^4$, we call the resulting Stein surface the Stein trace of the Legendrian knot.
The purpose of this section is to provide the first construction of exotic Stein traces.

In the previous section, we constructed an interesting family of exotic $4$-manifolds $X_0(J_k^m)$, $k \in \Z$ and $m$ sufficiently large.
It is a natural question then to ask if we can upgrade this to a family of exotic Stein fillings.
To do this, recall that $X_0(J_0^m)$ arise by blowing up and decreasing the $-2$ and $-1$ framings to $-3$ and $-1-m$ in Figure \ref{fig:kirby_trace_ribbon}.
Note that these $2$-handles are attached to $T \times D^2$ which admits a Stein structure as the disk cotangent bundle of a torus.
The $-3$ framed handle attachment can be made a Stein handle attachment easily by examining the Stein Kirby diagram of $T \times D^2$.
By taking $m$ large enough, we can guarantee that the $-1-m$ framing is sufficiently negative to be a Stein handle attachment as well.
We conclude that $X_0(J_0^m)$ is Stein.
Now $X_0(J_k^m)$ is obtained by a torus surgery on $X_0(J_0^m)$ on the symplectically embedded Stein $T \times D^2$ contained within it.
Since the zero section of $T \times D^2$ is Lagrangian, we can take this torus surgery to be a Luttinger surgery preserving the contact boundary.
We have now realized the exotic traces $X_0(J_k^m)$ as an infinite exotic family of Stein surfaces that fill the same contact $3$-manifold.

The discussion above realizes our exotic traces as exotic Stein fillings.
However, it does not realize our exotic traces as the Stein traces of Legendrian knots.
To do this we use a contact analogue of annulus twisting introduced by Casals, Etnyre, and Kegel.
They used this to construct an infinite collection of Legendrian knots that share a Stein trace \cite{contact_ann_twist}.
We now use their construction to compliment their result to construct an infinite collection of Legendrian knots that have homeomorphic, but non-diffeomorphic Stein traces.
First though we need to recall what we need of Casals, Etnyre, and Kegels' construction.
\begin{mydef}
Let $\ell_0$ be a Legendrian knot in $(M,\xi)$.
A pre-Lagrangian annulus $A$ in $(M,\xi)$ in the knot type $\ell_0$ is any embedded annulus formed by flowing the Legendrian knot $\ell_0$ for a short time under a Reeb flow associated to a contact form for $(M,\xi)$.
\end{mydef}
From this we can define the contact analogue of an annulus presentation.
\begin{mydef}
A contact annulus presentation $(A,\gamma)$ of a Legendrian knot $J_0$ in $S^3$ a diagram exhibiting $J_0$ as framed Legendrian push offs $\ell_0' \cup \ell_1'$ of a pre-lagrangian annulus $A$ embedded in $S^3$ banded together by a legendrian band $\gamma$.
\end{mydef}
Note that if we forget the contact structure, a contact annulus presentation is just a regular annulus presentation.
This then determines a family of knots and homeomorphisms identifying these knots' Dehn surgeries.
Casals, Etnyre, and Kegel show that then remembering the contact structure we get contactomorphisms of the contact $(-1)$-surgeries, i.e. the contact boundaries of the knots' Stein traces.
\begin{prop}(Theorem $3.6$ of \cite{contact_ann_twist})
Associated to a contact annulus presentation of a Legendrian knot $J_0$ there is an infinite family of Legendrian knots $\{J_k\}_{k\in \Z}$ such that the contact $(-1)$-surgeries are contactomorphic.
\end{prop}
It is clear from the proof, though not explicitly stated by Casals, Etnyre, and Kegel, that the resulting Legendrian knots' topological knot type is the same as that arises from annulus twisting while ignoring the contact structures.
We are now ready to upgrade our exotic traces into exotic Stein traces.

%

\begin{thm}
There exists an infinite family of Legendrian knots related by contact annulus twisting whose Stein traces are homeomorphic, but non-diffeomorphic Stein fillings of the same contact $3$-manifold.
\end{thm}
\begin{proof}
We show that the exotic traces of Theorem \ref{thm:ex_trace_full} can be modified appropriately.
To do this we mimic the second paragraph of this section to modify the annulus presentation generating our exotic traces into a contact annulus presentation.
The annulus presentation generating the exotic traces is shown in Figure \ref{fig:blow_up_ann_twist}.
The annulus in that presentation is described by the $-3$ framing on the unknot.
We can easily find a Legendrian unknot that describes a contact annulus in this topological type by adding negative stabilizations to the standard Legendrian unknot.
Now we need to ensure that the band in Figure \ref{fig:blow_up_ann_twist} can be taken to be Legendrian.
Recall that the $m$ left handed twists in that band were added so that there were sufficiently many to guarantee that $S^3_0(J_k^m)$ was asymmetric.
Take any Legendrian arc following the band, if the induced framing has enough twists to ensure asymmetry, we are done.
If not, add negative stabilizations to the Legendrian arc to add sufficiently many left handed twists.
Now we have a contact version of the annulus presentation generating the exotic traces of Theorem \ref{thm:ex_trace_full}.
The resulting Legendrian knots have exotic Stein traces which have the same contact boundary.
\end{proof}
%

\begin{cor}
There exist infinitely many homeomorphic, but non-diffeomorphic Stein surfaces with $b_2 =1$ that are Stein fillings of the same contact $3$-manifold.
\qed
\end{cor}
\section{A hidden fishtail symmetry}
A fishtail neighborhood $F$ is a basic building block of $4$-manifolds.
It can be described as $T \times D^2$ with a $2$-handle attached to an essential circle $\alpha$ of the torus $S = T \times \{\theta\}$ with framing $-1$ relative to the product framing.
Let $\tau^k$ be a $k$-fold torus twist of $\partial (T \times D^2)$ along the torus $S$ parallel to $\alpha$.
Suppose $F$ is embedded in some $4$-manifold $X$ and so we can form the torus surgery $X_{T,\tau^k}$ on $T$ contained within $F$.
It turns out that there is a diffeomorphism $X_{T,\tau^k} \cong X$ called a fishtail symmetry.
Fishtail symmetries originally arose through Moishezon's classification of elliptic surfaces \cite{moish}.
Gompf re-contextualized Akbulut's proof that the main family of Cappell-Shaneson spheres are standard\cite{CS_standard} using fishtail symmetries \cite{More_Cappell}.
This allowed Gompf to simplify and generalize prior work on the Cappell-Shaneson spheres.
The main observation of this section is the presence of a hidden fishtail symmetry in previous constructions of trace diffeomorphisms.
Using this we will construct new potential counterexamples to the smooth $4$-dimensional Poincar\'e Conjecture.
First though, we reexamine certain trace diffeomorphisms using fishtail symmetries.
\begin{prop}
Suppose $\anntwifam$ has annulus presentation $(A,\gamma)$ where  $A$ is an annulus described by $-1$ framing on the unknot.
Then $\phi_k$ extends to a trace diffeomorphism $\Phi_k:X_0(J_0) \cong X_0(J_k)$.
\label{prop:fishtail}
\end{prop}
\begin{proof}
Applying the proof of Theorem \ref{thm:ann_twi_log_trans} to such an annulus presentation of $J_0$ to get a diagram like Figure \ref{fig:ann_log_3} with $n=-1$.
We have a diagram of $T \times D^2$ with two cancelling $2$-handles.
That theorem says that annulus twisting is the same as a torus surgery on $T$ with homeomorphism a $k$-fold torus twist $\tau^k$ on $S = T \times \{\theta\}$ in the direction inherited from $A$.
The condition on $A$ implies that we have a $2$-handle $H$ attached to a curve $\alpha \subset S$ running parallel to $A$ with framing $-1$ relative to the product framing.
This $H$ along with $T \times D^2$ form a fishtail neighborhood in our diagram.
Using this, we have to show that $\tau^k$ on $S$ can be realized by a diffeomorphism of $X_0(J_0)$.

\begin{figure}
\centering
\subfloat[]{{
    \fontsize{10pt}{12pt}\selectfont
    \def\svgwidth{2in}
    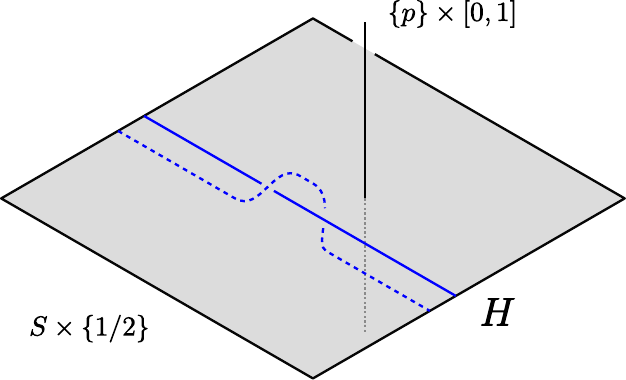
}
\label{fig:fishtail_symm_1}
} \ \
\subfloat[]{{
    \fontsize{10pt}{12pt}\selectfont
    \def\svgwidth{2in}
    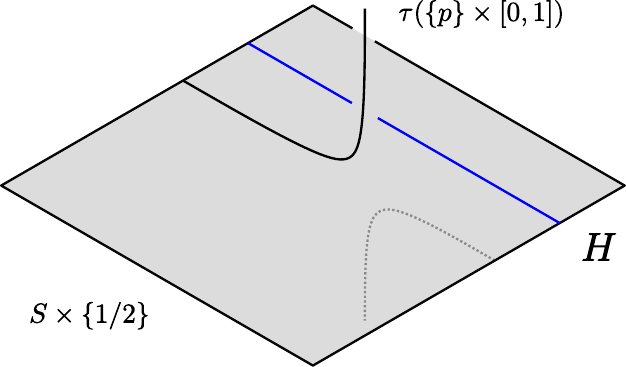
}
\label{fig:fishtail_symm_2}
}
\caption{Using a fishtail to realize a torus twist}
\label{fig:fishtail_symm}
\end{figure}

Parametrize a neighborhood $S \times [0,1]$ of $S$ in $\partial (T \times D^2)$ and take the $2$-handle $H$ to lay at $\alpha \times \{1/2\}$.
This is depicted in Figure \ref{fig:fishtail_symm_1}, the gray rhombus with opposite edges identified is the torus $S \times \{1/2\}$.
The $2$-handle $H$ is shown in blue with a dashed blue framing curve.
For each $p \in S$, the torus twist $\tau$ wraps each $\{p\} \times [0,1]$ once around $\alpha$.
Now push $H$ around $S \times \{1/2\}$ until it comes back to where it started.
When it meets each $\{p\} \times [0,1]$, slide $\{p\} \times [0,1]$ over $H$ by banding it to the dashed blue framing curve to get Figure \ref{fig:fishtail_symm_2}.
This resolves the intersection with $H$, pushing $H$ past, and wrapping $\{p\} \times [0,1]$ around $\alpha$ sending it to $\tau(\{p\} \times [0,1])$ as desired.

\begin{figure}
\centering
\subfloat[]{{
    \fontsize{10pt}{12pt}\selectfont
    \def\svgwidth{2in}
    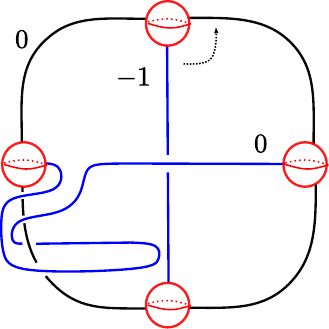
}
\label{fig:fishtail_trace_1}
} \ \
\subfloat[]{{
    \fontsize{10pt}{12pt}\selectfont
    \def\svgwidth{2in}
    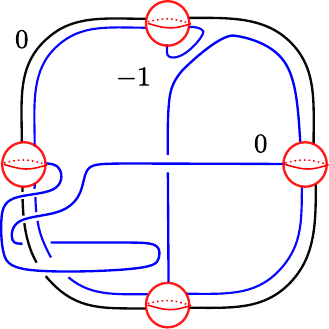
}
\label{fig:fishtail_trace_2}
}

\subfloat[]{{
    \fontsize{10pt}{12pt}\selectfont
    \def\svgwidth{2in}
    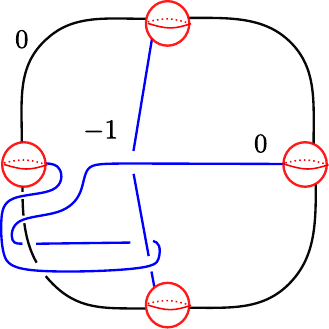
}
\label{fig:fishtail_trace_3}
} \ \
\subfloat[]{{
    \fontsize{10pt}{12pt}\selectfont
    \def\svgwidth{2in}
    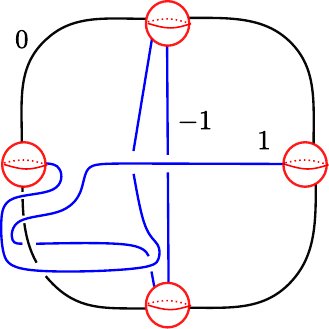
}}
\label{fig:fishtail_trace_4}
\caption{Seeing fishtail symmetry diagrammatically}
\label{fig:fishtail_trace}
\end{figure}

We can also do this diagrammatically, in Figure \ref{fig:fishtail_trace_1} we have drawn Figure \ref{fig:ann_log_3} with $n = -1$ to depict this.
We first push the $-1$ framed $2$-handle $H$ around $S$ until we see another $2$-handle.
Do this by making the indicated slide to Figure \ref{fig:fishtail_trace_2} which is isotopic to Figure \ref{fig:fishtail_trace_3}.
Now slide the other $2$-handle over $H$, this resolves the intersection with $H$, and wraps the other $2$-handle around the vertical $1$-handle.
The $2$-handle $H$ has returned to its initial position while the other $2$-handle has been twisted along the torus $S$.
\end{proof}
This gives conditions for when an annulus twist homeomorphism can extend to a trace diffeomorphism.
This was already known by work of Abe-Jong-Omae-Takeuchi \cite{ann_diff} using the Akbulut trick \cite{twodim} which is referred to as Property $U$ by Manolescu and Piccirillo \cite{zero_surg_exotic}.
Here though we recognize this diffeomorphism as a fishtail symmetry.
As an immediate corollary, we see that this trace diffeomorphism identifies the tori associated to the torus surgery realizing the annulus twist.
\begin{cor}
Let $\Phi_k: X_0(J_0) \cong X_0(J_k)$ be a trace diffeomorphism as in Proposition \ref{prop:fishtail}.
Then $\Phi_k$ identifies the tori $T_0 \subset X_0(J_0)$ and $T_k \subset X_0(J_k)$ determined by the annulus presentation as in Theorem \ref{thm:ann_twi_log_trans}.
\qed
\end{cor}
This is difficult to prove when constructing the trace diffeomorphism using the Akbulut trick as in \cite{ann_diff}.
It is not clear if the trace diffeomorphism coming from Akbulut's trick are the same as the fishtail symmetry and if the above corollary holds for those trace diffeomorphisms.
However, constructing the trace diffeomorphism as a fishtail symmetry makes it obvious.
This is relevant to a question of Casals, Etnyre, and Kegel asking if certain 
equivalence of Stein traces they constructed identify Lagrangian tori \cite{contact_ann_twist}.
The fishtail symmetry approach should be able to help answer the question of Casals-Etnyre-Kegal.
However, they use Akbulut's trick to construct their equivalences of Stein traces and so as noted previously, it is difficult to tell if this approach is applicable.

\begin{figure}
\centering
\subfloat[]{{
    \fontsize{10pt}{12pt}\selectfont
    \def\svgwidth{2in}
\begingroup%
  \makeatletter%
  \providecommand\color[2][]{%
    \errmessage{(Inkscape) Color is used for the text in Inkscape, but the package 'color.sty' is not loaded}%
    \renewcommand\color[2][]{}%
  }%
  \providecommand\transparent[1]{%
    \errmessage{(Inkscape) Transparency is used (non-zero) for the text in Inkscape, but the package 'transparent.sty' is not loaded}%
    \renewcommand\transparent[1]{}%
  }%
  \providecommand\rotatebox[2]{#2}%
  \newcommand*\fsize{\dimexpr\f@size pt\relax}%
  \newcommand*\lineheight[1]{\fontsize{\fsize}{#1\fsize}\selectfont}%
  \ifx\svgwidth\undefined%
    \setlength{\unitlength}{216bp}%
    \ifx\svgscale\undefined%
      \relax%
    \else%
      \setlength{\unitlength}{\unitlength * \real{\svgscale}}%
    \fi%
  \else%
    \setlength{\unitlength}{\svgwidth}%
  \fi%
  \global\let\svgwidth\undefined%
  \global\let\svgscale\undefined%
  \makeatother%
  \begin{picture}(1,1)%
    \lineheight{1}%
    \setlength\tabcolsep{0pt}%
    \put(0,0){\includegraphics[width=\unitlength,page=1]{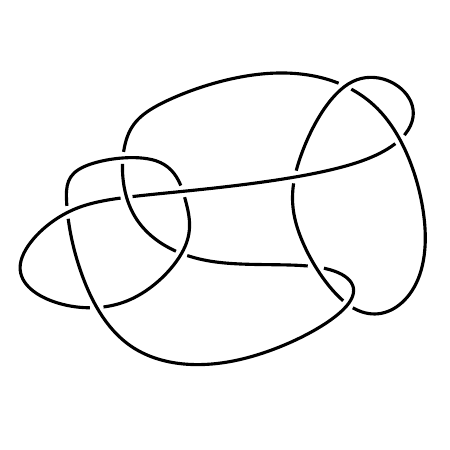}}%
    \put(0.21421929,0.72478061){\color[rgb]{0,0,0}\makebox(0,0)[lt]{\lineheight{1.25}\smash{\begin{tabular}[t]{l}$0$\end{tabular}}}}%
    \put(0,0){\includegraphics[width=\unitlength,page=2]{conway_1.pdf}}%
  \end{picture}%
\endgroup%

}
\label{fig:conway_1}
} \ \
\subfloat[]{{
    \fontsize{10pt}{12pt}\selectfont
    \def\svgwidth{2in}
\begingroup%
  \makeatletter%
  \providecommand\color[2][]{%
    \errmessage{(Inkscape) Color is used for the text in Inkscape, but the package 'color.sty' is not loaded}%
    \renewcommand\color[2][]{}%
  }%
  \providecommand\transparent[1]{%
    \errmessage{(Inkscape) Transparency is used (non-zero) for the text in Inkscape, but the package 'transparent.sty' is not loaded}%
    \renewcommand\transparent[1]{}%
  }%
  \providecommand\rotatebox[2]{#2}%
  \newcommand*\fsize{\dimexpr\f@size pt\relax}%
  \newcommand*\lineheight[1]{\fontsize{\fsize}{#1\fsize}\selectfont}%
  \ifx\svgwidth\undefined%
    \setlength{\unitlength}{216bp}%
    \ifx\svgscale\undefined%
      \relax%
    \else%
      \setlength{\unitlength}{\unitlength * \real{\svgscale}}%
    \fi%
  \else%
    \setlength{\unitlength}{\svgwidth}%
  \fi%
  \global\let\svgwidth\undefined%
  \global\let\svgscale\undefined%
  \makeatother%
  \begin{picture}(1,1)%
    \lineheight{1}%
    \setlength\tabcolsep{0pt}%
    \put(0,0){\includegraphics[width=\unitlength,page=1]{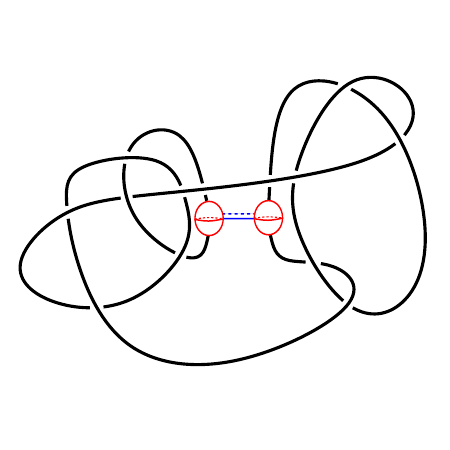}}%
    \put(0.21421929,0.72478061){\color[rgb]{0,0,0}\makebox(0,0)[lt]{\lineheight{1.25}\smash{\begin{tabular}[t]{l}$0$\end{tabular}}}}%
  \end{picture}%
\endgroup%

}
\label{fig:conway_2}
}

\subfloat[]{{
    \fontsize{10pt}{12pt}\selectfont
    \def\svgwidth{2in}
\begingroup%
  \makeatletter%
  \providecommand\color[2][]{%
    \errmessage{(Inkscape) Color is used for the text in Inkscape, but the package 'color.sty' is not loaded}%
    \renewcommand\color[2][]{}%
  }%
  \providecommand\transparent[1]{%
    \errmessage{(Inkscape) Transparency is used (non-zero) for the text in Inkscape, but the package 'transparent.sty' is not loaded}%
    \renewcommand\transparent[1]{}%
  }%
  \providecommand\rotatebox[2]{#2}%
  \newcommand*\fsize{\dimexpr\f@size pt\relax}%
  \newcommand*\lineheight[1]{\fontsize{\fsize}{#1\fsize}\selectfont}%
  \ifx\svgwidth\undefined%
    \setlength{\unitlength}{216bp}%
    \ifx\svgscale\undefined%
      \relax%
    \else%
      \setlength{\unitlength}{\unitlength * \real{\svgscale}}%
    \fi%
  \else%
    \setlength{\unitlength}{\svgwidth}%
  \fi%
  \global\let\svgwidth\undefined%
  \global\let\svgscale\undefined%
  \makeatother%
  \begin{picture}(1,1)%
    \lineheight{1}%
    \setlength\tabcolsep{0pt}%
    \put(0,0){\includegraphics[width=\unitlength,page=1]{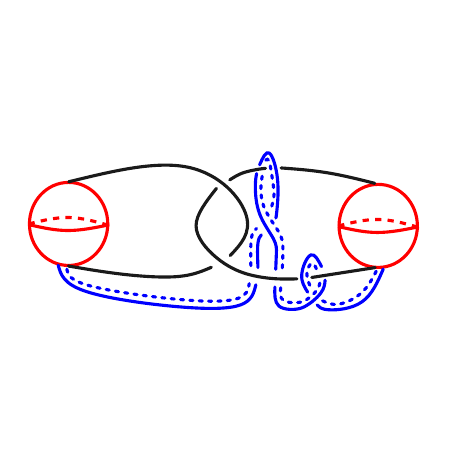}}%
    \put(0.2285046,0.64412169){\color[rgb]{0,0,0}\makebox(0,0)[lt]{\lineheight{1.25}\smash{\begin{tabular}[t]{l}$0$\end{tabular}}}}%
  \end{picture}%
\endgroup%

}
\label{fig:conway_3}
} \ \
\subfloat[]{{
    \fontsize{10pt}{12pt}\selectfont
    \def\svgwidth{2in}
    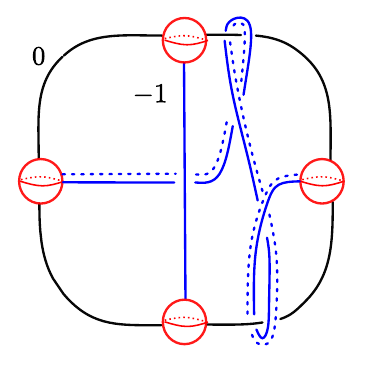
}
\label{fig:conway_4}
}

\caption{Fishtail symmetry in the Conway trace}
\label{fig:conway_fishtail}
\end{figure}
We now observe that there is a well hidden fishtail symmetry in Piccirillo's proof that the Conway knot is not slice \cite{Conway_knot}.
The key step is to construct a knot that shares a zero trace with the Conway knot.
This is crucial as it allows one to escape the difficulty of the Conway knot's sliceness and work with a different knot.
Piccirillo employs a clever Kirby calculus trick to construct this knot that shares a trace with the Conway knot.
This trick starts with a diagram for the Conway knot $C$ with the unknotting crossing highlighted as in Figure \ref{fig:conway_1}.
To view this as a fishtail symmetry, first add a pair of cancelling $1$ and $2$-handles that undoes the unknotting crossing like in Figure \ref{fig:conway_2}.
Here we use double strand notation to keep track of the framing of the $2$-handle.
Think of this as doing a sort of band move that undoes the unknotting crossing where each ball of the added $1$-handle is the result of doing the band move.
Now if we ignore the blue $2$-handle in Figure \ref{fig:conway_2}, it is easy to see that the diagram simplifies because we have undone the unknotting crossing in the black $2$-handle.
Doing this simplification to the black $2$-handle and the red $1$-handle, we get Figure \ref{fig:conway_3}.
Finally, we remove the central twist in the black $2$-handle by adding another pair of cancelling $1$ and $2$-handles.
This gives Figure \ref{fig:conway_4} with the fishtail neighborhood now clearly visible.

Now proceeding as in Proposition \ref{prop:fishtail} with a single torus twist, we get a knot trace diffeomorphic to the Conway knot trace.
This is Piccirillo's companion knot trace that she used to show that the Conway knot is not slice.
We now see that Piccirillo's proof that the Conway knot is not slice can instead be seen as a fishtail symmetry.
However, it is still somewhat mysterious the presence of the fishtail symmetry; popping up here as well as in the work of Gompf and Akbulut on the Cappell-Shaneson spheres \cite{akb_kir_standard,CS_standard,More_Cappell}.
Nevertheless, once we recognize an interesting torus to do surgery with, we should not just leave it be.
Instead, we look for other interesting surgeries to do with this torus.
From this we find new potential counterexamples to SPC4.
\begin{thm}
There is an infinite collection of $4$-manifolds $C_k$ such that if any embed in the $4$-sphere, then an exotic $4$-sphere exists.
\label{thm:SPC4_conway}
\end{thm}
\begin{proof}
%
Let $C_k$ be the result of torus surgery on $X_0(C)$ using the torus apparent in Figure \ref{fig:conway_4} and a $k$-fold torus twist perpendicular to the fishtail symmetry.
These produce simply connected $4$-manifolds $C_k$ with the same boundary and homology as $X_0(C)$.
Suppose some $C_k$ embeds smoothly in $S^4$, denote its exterior by $E(C_k)$.
Since $C_k$ is formed from $0$, $1$, and $2$-handles, this decomposes $S^4$ as $E(C_k) \cup 2-handles \cup 3-handles \cup 4-handle$.
This implies that the boundary $\partial E(C_k) = S^3_0(C)$ normally generates the fundamental group of $E(C_k)$.
Otherwise, we would not get a simply connected fundamental group.
Now consider the space $X_k = E(C_k) \cup X_0(C)$.
By Seifert-Van Kampen, the fundamental group of $X_k$ is the amalgamated product $\pi_1(E(C_k)) *_{\pi_1(S^3_0(C))} \pi_1(X_0(C))$.
The latter term in this expression is trivial which trivializes the $\pi_1(S^3_0(C))$ we are taking the product over.
As noted earlier, $\pi_1(S^3_0(C))$ normally generates $E(C_k)$ and so this also kills the $\pi_1(E(C_k))$ term as well.
Therefore $\pi_1(X_k)$ is trivial and we can conclude that $X_k$ is homeomorphic to $S^4$ after an easy homology computation.
However, $X_k$ contains a smoothly embedded $X_0(C)$ while the $4$-sphere does not as Piccirillo showed \cite{Conway_knot}.
We conclude that such an $X_k$ is not diffeomorphic to the standard $4$-sphere and is therefore an exotic $4$-sphere.
\end{proof}
Manolescu and Piccirillo found five knots that if any were slice, then one could construct an exotic $S^4$ \cite{zero_surg_exotic}.
The author had previously shown that these knots were not slice and hence ruled them out as counterexamples to SPC4 \cite{trace_emb_zero}.
Theorem \ref{thm:SPC4_conway} should be thought of as an attempt to re-implement their strategy.
Manolescu and Piccirillo wanted to remove the trace of a slice knot from $S^4$ and replace it with the trace of a non-slice knot.
Instead, we propose to remove a $4$-manifold with the same boundary and algebraic topology as the trace of a non-slice knot and then glue in that non-slice trace.
We believe that this modification of the Manolescu-Piccirillo approach is stronger for several reasons.

First, the obstruction to being standard would come from the non-sliceness of the Conway knot.
As far as the author is aware, the Conway knot is the smallest crossing knot that is known to not be slice, but it is currently unknown whether it is slice in a homotopy $4$-ball.
Much effort had been spent on the Conway knot's slice properties, but Piccirillo only showed that it is not slice in the standard $4$-ball.
This makes it a natural candidate to be exotically slice and potentially useful to disprove SPC4.
Furthermore, Mark Hughes constructed a neural network that attempted to predict the sliceness of knots and it gave an approximately $ \% 50$ chance that the Conway knot is slice \cite{neural}.
To the neural network, the Conway knot only seems ``somewhat slice'' and this might indicate it being homotopy slice.

The $4D$ perspective also sheds light on the strength of these new potential counterexamples to SPC4.
This would essentially be a convoluted way to find a torus surgery on the $4$-sphere that gives rise to an exotic $4$-sphere.
Essentially all known examples of exotic $4$-manifolds come from torus surgeries \cite{round_handles}.
So if one would like to use knot traces and concordance to disprove S4PC, it would be best that the resulting exotic $4$-sphere is manifestly a torus surgery.
Moreover, the trefoil is unknotting number one and so we can apply the above construction.
There we recover the zero trace of the trefoil as its usual diagram as a cusp neighborhood in an elliptic surface where there are two orthogonal fishtail neighborhoods.
If we wish to do a torus surgery on the trace of the trefoil to get something interesting, the fishtail symmetries tell us what not to do.
The only possibility to change the topology is to do what is usually called a logarithmic transform.
This is a torus surgery on $T \times D^2$ by a homeomorphism $\phi: T^2 \times S^1 \rightarrow T^2 \times S^1$ that wraps $\{pt\}  \times S^1$ around itself non-trivially.
To be precise, this means that $\phi$ on homology sends $[\{pt\}  \times S^1]$ to $p[\{pt\}  \times S^1] + q [\theta \times S^1 \times \theta'] + r [S^1 \times \theta \times \theta']$ with $p \neq \pm 1$.
This $p$ is called the multiplicity of the logarithmic transform or torus surgery.
The torus surgeries in Theorem \ref{thm:SPC4_conway} all had multiplicity $1$.
We would now like to mimic this picture for the trefoil with the Conway knot.
\subsection{A second family}
\begin{figure}
\centering
\subfloat[]{{
    \fontsize{10pt}{12pt}\selectfont
    \def\svgwidth{2in}
\begingroup%
  \makeatletter%
  \providecommand\color[2][]{%
    \errmessage{(Inkscape) Color is used for the text in Inkscape, but the package 'color.sty' is not loaded}%
    \renewcommand\color[2][]{}%
  }%
  \providecommand\transparent[1]{%
    \errmessage{(Inkscape) Transparency is used (non-zero) for the text in Inkscape, but the package 'transparent.sty' is not loaded}%
    \renewcommand\transparent[1]{}%
  }%
  \providecommand\rotatebox[2]{#2}%
  \newcommand*\fsize{\dimexpr\f@size pt\relax}%
  \newcommand*\lineheight[1]{\fontsize{\fsize}{#1\fsize}\selectfont}%
  \ifx\svgwidth\undefined%
    \setlength{\unitlength}{180bp}%
    \ifx\svgscale\undefined%
      \relax%
    \else%
      \setlength{\unitlength}{\unitlength * \real{\svgscale}}%
    \fi%
  \else%
    \setlength{\unitlength}{\svgwidth}%
  \fi%
  \global\let\svgwidth\undefined%
  \global\let\svgscale\undefined%
  \makeatother%
  \begin{picture}(1,1)%
    \lineheight{1}%
    \setlength\tabcolsep{0pt}%
    \put(0,0){\includegraphics[width=\unitlength,page=1]{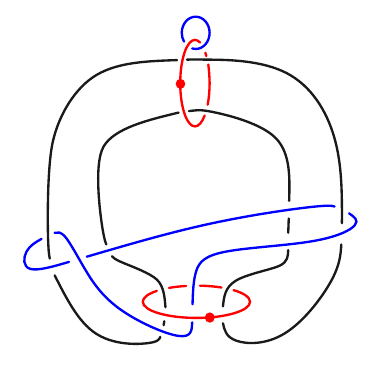}}%
    \put(0.57928245,0.92522482){\color[rgb]{0,0,1}\makebox(0,0)[lt]{\lineheight{1.25}\smash{\begin{tabular}[t]{l}$-1$\end{tabular}}}}%
    \put(0.47207464,0.41981675){\color[rgb]{0,0,1}\makebox(0,0)[lt]{\lineheight{1.25}\smash{\begin{tabular}[t]{l}$0$\end{tabular}}}}%
    \put(0.85676136,0.79279167){\color[rgb]{0.10196078,0.10196078,0.10196078}\makebox(0,0)[lt]{\lineheight{1.25}\smash{\begin{tabular}[t]{l}$0$\end{tabular}}}}%
  \end{picture}%
\endgroup%

}
\label{fig:log_conway_1}
} \ \
\subfloat[]{{
    \fontsize{10pt}{12pt}\selectfont
    \def\svgwidth{2in}
\begingroup%
  \makeatletter%
  \providecommand\color[2][]{%
    \errmessage{(Inkscape) Color is used for the text in Inkscape, but the package 'color.sty' is not loaded}%
    \renewcommand\color[2][]{}%
  }%
  \providecommand\transparent[1]{%
    \errmessage{(Inkscape) Transparency is used (non-zero) for the text in Inkscape, but the package 'transparent.sty' is not loaded}%
    \renewcommand\transparent[1]{}%
  }%
  \providecommand\rotatebox[2]{#2}%
  \newcommand*\fsize{\dimexpr\f@size pt\relax}%
  \newcommand*\lineheight[1]{\fontsize{\fsize}{#1\fsize}\selectfont}%
  \ifx\svgwidth\undefined%
    \setlength{\unitlength}{180bp}%
    \ifx\svgscale\undefined%
      \relax%
    \else%
      \setlength{\unitlength}{\unitlength * \real{\svgscale}}%
    \fi%
  \else%
    \setlength{\unitlength}{\svgwidth}%
  \fi%
  \global\let\svgwidth\undefined%
  \global\let\svgscale\undefined%
  \makeatother%
  \begin{picture}(1,1)%
    \lineheight{1}%
    \setlength\tabcolsep{0pt}%
    \put(0,0){\includegraphics[width=\unitlength,page=1]{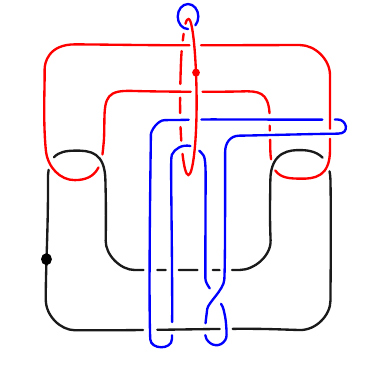}}%
    \put(0.55103662,0.93820745){\color[rgb]{0,0,1}\makebox(0,0)[lt]{\lineheight{1.25}\smash{\begin{tabular}[t]{l}$-1$\end{tabular}}}}%
    \put(0.9045919,0.59165958){\color[rgb]{0,0,1}\makebox(0,0)[lt]{\lineheight{1.25}\smash{\begin{tabular}[t]{l}$0$\end{tabular}}}}%
    \put(0.90013267,0.79551126){\color[rgb]{1,0,0}\makebox(0,0)[lt]{\lineheight{1.25}\smash{\begin{tabular}[t]{l}$0$\end{tabular}}}}%
  \end{picture}%
\endgroup%

}
\label{fig:log_conway_2}
}

\subfloat[]{{
    \fontsize{10pt}{12pt}\selectfont
    \def\svgwidth{2in}
\begingroup%
  \makeatletter%
  \providecommand\color[2][]{%
    \errmessage{(Inkscape) Color is used for the text in Inkscape, but the package 'color.sty' is not loaded}%
    \renewcommand\color[2][]{}%
  }%
  \providecommand\transparent[1]{%
    \errmessage{(Inkscape) Transparency is used (non-zero) for the text in Inkscape, but the package 'transparent.sty' is not loaded}%
    \renewcommand\transparent[1]{}%
  }%
  \providecommand\rotatebox[2]{#2}%
  \newcommand*\fsize{\dimexpr\f@size pt\relax}%
  \newcommand*\lineheight[1]{\fontsize{\fsize}{#1\fsize}\selectfont}%
  \ifx\svgwidth\undefined%
    \setlength{\unitlength}{180bp}%
    \ifx\svgscale\undefined%
      \relax%
    \else%
      \setlength{\unitlength}{\unitlength * \real{\svgscale}}%
    \fi%
  \else%
    \setlength{\unitlength}{\svgwidth}%
  \fi%
  \global\let\svgwidth\undefined%
  \global\let\svgscale\undefined%
  \makeatother%
  \begin{picture}(1,1)%
    \lineheight{1}%
    \setlength\tabcolsep{0pt}%
    \put(0,0){\includegraphics[width=\unitlength,page=1]{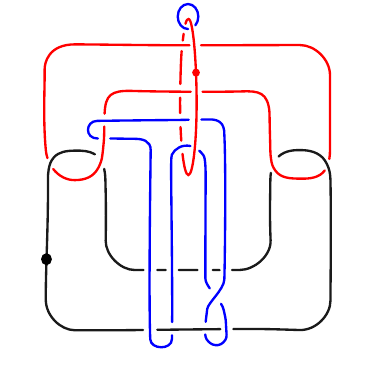}}%
    \put(0.55103662,0.93820745){\color[rgb]{0,0,1}\makebox(0,0)[lt]{\lineheight{1.25}\smash{\begin{tabular}[t]{l}$-1$\end{tabular}}}}%
    \put(0.30959975,0.46807447){\color[rgb]{0,0,1}\makebox(0,0)[lt]{\lineheight{1.25}\smash{\begin{tabular}[t]{l}$0$\end{tabular}}}}%
    \put(0.90013267,0.79551126){\color[rgb]{1,0,0}\makebox(0,0)[lt]{\lineheight{1.25}\smash{\begin{tabular}[t]{l}$0$\end{tabular}}}}%
  \end{picture}%
\endgroup%

}
\label{fig:log_conway_3}
} \ \
\subfloat[]{{
    \fontsize{10pt}{12pt}\selectfont
    \def\svgwidth{2in}
\begingroup%
  \makeatletter%
  \providecommand\color[2][]{%
    \errmessage{(Inkscape) Color is used for the text in Inkscape, but the package 'color.sty' is not loaded}%
    \renewcommand\color[2][]{}%
  }%
  \providecommand\transparent[1]{%
    \errmessage{(Inkscape) Transparency is used (non-zero) for the text in Inkscape, but the package 'transparent.sty' is not loaded}%
    \renewcommand\transparent[1]{}%
  }%
  \providecommand\rotatebox[2]{#2}%
  \newcommand*\fsize{\dimexpr\f@size pt\relax}%
  \newcommand*\lineheight[1]{\fontsize{\fsize}{#1\fsize}\selectfont}%
  \ifx\svgwidth\undefined%
    \setlength{\unitlength}{180bp}%
    \ifx\svgscale\undefined%
      \relax%
    \else%
      \setlength{\unitlength}{\unitlength * \real{\svgscale}}%
    \fi%
  \else%
    \setlength{\unitlength}{\svgwidth}%
  \fi%
  \global\let\svgwidth\undefined%
  \global\let\svgscale\undefined%
  \makeatother%
  \begin{picture}(1,1)%
    \lineheight{1}%
    \setlength\tabcolsep{0pt}%
    \put(0,0){\includegraphics[width=\unitlength,page=1]{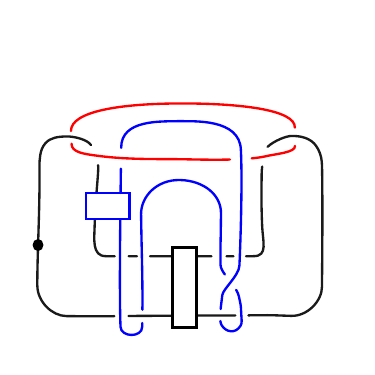}}%
    \put(0.65497865,0.75578087){\color[rgb]{1,0,0}\makebox(0,0)[lt]{\lineheight{1.25}\smash{\begin{tabular}[t]{l}$0$\end{tabular}}}}%
    \put(0.47026359,0.29175892){\color[rgb]{0,0,0}\rotatebox{-90}{\makebox(0,0)[lt]{\lineheight{1.25}\smash{\begin{tabular}[t]{l}$-1$\end{tabular}}}}}%
    \put(0.25850879,0.43569484){\color[rgb]{0,0,1}\makebox(0,0)[lt]{\lineheight{1.25}\smash{\begin{tabular}[t]{l}$p$\end{tabular}}}}%
    \put(0.4389182,0.4521485){\color[rgb]{0,0,1}\makebox(0,0)[lt]{\lineheight{1.25}\smash{\begin{tabular}[t]{l}$p$\end{tabular}}}}%
  \end{picture}%
\endgroup%

}
\label{fig:log_conway_4}
}

\caption{Logarithmic transform on the Conway trace}
\label{fig:log_conway}
\end{figure}
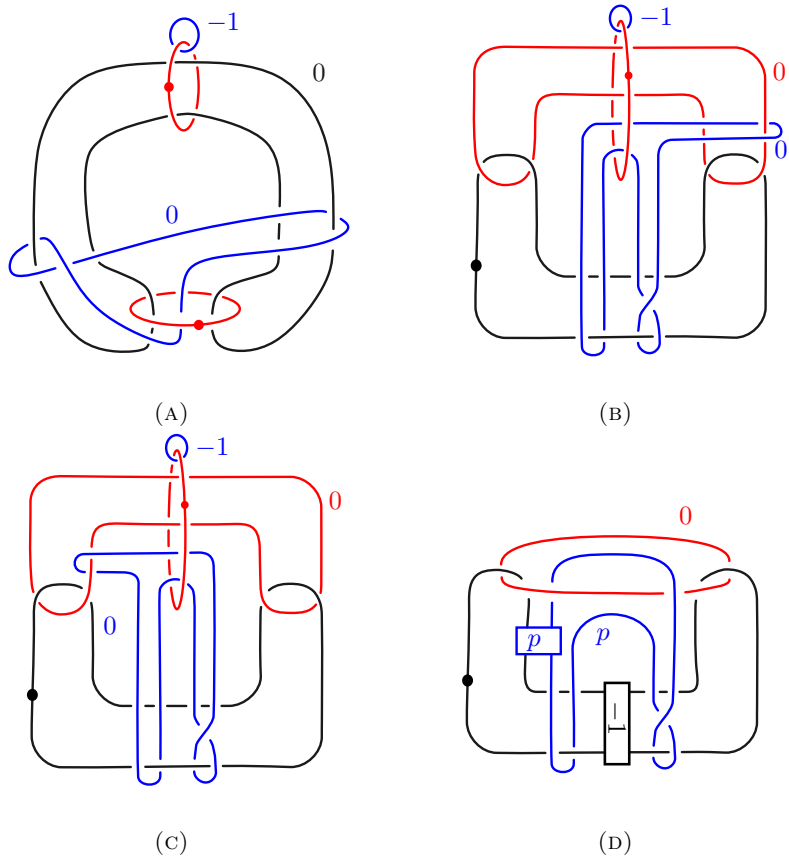
In essence, we would like for an analogy to hold.
We hope that the Conway knot trace can be used for the $4$-sphere in a similar way that the trefoil trace is for elliptic surfaces.
We make the analogy more explicit by considering the $p$-fold logarithmic transform of the Conway knot trace.
To do so, we follow Gompf-Stipsicz Section $8.3$; in particular, Figures $8.25$ and $8.26$ and the discussion surrounding them \cite{4_Mfld_Kirby}.
First, change to dotted circle notation to get Figure \ref{fig:log_conway_1}.
Then, do a zero dot swap on the black $2$-handle and lower red $1$-handle to do a multiplicity zero logarithmic transform.
We can do an isotopy of the diagram to bring the black dotted circle down and the red zero framed $2$-handle up.
This is best done by first swinging the top red dotted circle down and then rotating the whole diagram $180^\circ$ on the page to get Figure \ref{fig:log_conway_2}.
Now pull the red $2$-handle down, bringing the linking with the blue $2$-handle in, and then move the linking with blue over to the left to get Figure \ref{fig:log_conway_3}.
To change from our current picture of a multiplicity zero log transform to multiplicity $p$, we perform a $p$-fold torus twist.
There is a torus in Figure \ref{fig:log_conway_3} formed by taking the sphere spanned by the red zero framed $2$-handle and tubing over the top half arc of the black dotted circle.
The torus twist wraps the blue $2$-handle around the black dotted circle and increases its framing by $p$.
After cancelling the red dotted circle against its blue meridian and bringing the resulting twists down to the bottom, we get Figure \ref{fig:log_conway_4}.
Let the resulting $4$-manifold be denoted by $Y_p$.

The $4$-manifolds $Y_p$ depicted by Figure \ref{fig:log_conway_4} are not potential counterexamples to SPC4 in the manner of Theorem \ref{thm:SPC4_conway}.
Unfortunately, the non-trivial multiplicity of the logarithmic transform provides homological obstructions to embedding $Y_p$ in $S^4$ when $p \neq \pm 1$.
Despite not fitting neatly into the form of Theorem \ref{thm:SPC4_conway}, these $Y_p$ provide some advantages that may still be useful even when $p \neq \pm 1$.
First, the diagrams are much simpler and also admit an enticing almost symmetry.
This is in addition to the points discussed above and the desired analogy with the trefoil trace in the elliptic surface.
The $4$-manifolds $Y_p$ could still be useful in the following ways:
\begin{itemize}
    \item If one were to cap off $Y_p$ with $n-1$ $2$-handles, $n$ $3$-handles and a $4$-handle and then transfer those handle attachments to the Conway knot trace.
    This would then give an exotic $4$-sphere as the resulting $4$-manifold would have no $1$-handles and is simply connected with the right homology.
    \item If we do not concern ourselves with maintaining simply connectedness, then these $Y_p$ may be useful for showing that the Conway knot is rationally slice.
    If some $Y_p$ embedded in $S^4$, then replacing it with the Conway knot trace amounts to a torus surgery on the $4$-sphere.
    As long as the multiplicity of this torus surgery is non-vanishing, then the resulting $4$-manifold is a rational homology $4$-sphere.
    This is roughly analogous to non-zero framings on a knot surger to a rational homology $3$-sphere, see Section $2.3$ of Larson \cite{larson_tori}.
    Then the resulting rational homology $4$-sphere would contain the Conway knot trace.
    Then by the trace embedding lemma, the Conway knot would be slice in a rational homology $4$-ball. 
\end{itemize}
%

\section{Motivation and Further Questions}
\label{sec:questions}
\subsection{Slice Akbulut-Kirby Conjecture}
The original Akbulut-Kirby Conjecture is that if two knots have the same zero surgery, then they are concordant, i.e. cobound a smooth annulus in $S^3 \times I$ (Problem $1.19$ of Kirby's problem list \cite{prob_list}).
Yasui disproved the Akbulut-Kirby Conjecture by using zero traces of knots \cite{yasui_akb_kirby_conj}.
However, this left the slice case of the Akbulut-Kirby Conjecture open as none of Yasui's counterexamples were slice knots.
\begin{conj}[Slice Akbulut-Kirby Conjecture]
\label{conj:slice_akb_kirb}
If a knot $K$ is slice and there is a zero surgery homeomorphism $\zsg$, then $K'$ is also slice.
\end{conj}
Conjecture \ref{conj:slice_akb_kirb} has been considered for a long time, but never seems to have been given a name.
Let us call it the slice case of the Akbulut-Kirby Conjecture.
It was noted in the Kirby problem list that the Akbulut-Kirby Conjecture holds when one knot is slice if SPC4 holds.
This is because $K'$ would be slice in the associated slice trace surgery which would then be standard assuming SPC4.
Manolescu and Piccirillo hoped to disprove SPC4 by finding a counterexample to the Slice Akbulut-Kirby Conjecture \cite{zero_surg_exotic}.
In other words, they hoped to construct an exotic $4$-sphere as a strongly exotic slice trace surgery using the terminology from Section \ref{sec:MP_spheres} that is detected by the $s$-invariant.
To get partial progress towards this method of disproving SPC4, the original goal of this paper was to construct a strongly exotic $H$-slice trace surgery on an elliptic surface.
This was unsuccessful as we only constructed weakly exotic $H$-slice trace surgeries.
These are exotic $H$-slice trace surgeries $X_{(D,\phi)}$ with knots $(K,K')$ where both $K$ and $K'$ are $H$-slice in $X$.
This means $H$-sliceness of $K'$ does not smoothly distinguish $X$ from $X_{(D,\phi)}$.

\begin{figure}
\centering
    \fontsize{10pt}{12pt}\selectfont
    \def\svgwidth{4in}
    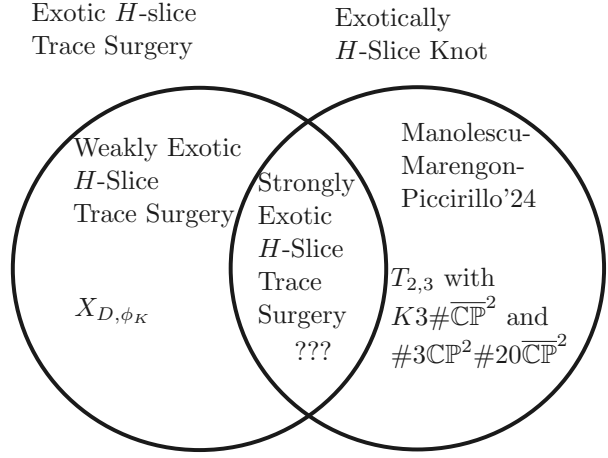
\caption{Venn diagram illustrating the relation between exotically H-slice knots and exotic H-slice trace surgeries}
\label{fig:venn_diag}
\end{figure}
Recall that a knot $K$ is exotically $H$-slice in a $4$-manifold $X$ with respect to another $4$-manifold $X'$, if $X$ and $X'$ are homeomorphic, but $K$ is $H$-slice in $X'$ and not in $X'$.
This smoothly distinguishes $X$ and $X'$ and implies that they are an exotic pair of $4$-manifolds.
Manolescu-Marengon-Piccirillo showed that the right hand trefoil is exotically $H$-slice in $K3 \# \CPbar$ with respect to $3 \CP \# 20 \CPbar$ \cite{rel_gen_indef_manifold}.
Despite having examples of exotically $H$-slice knots and exotic $H$-slice trace surgeries, an example with both remains elusive.
This situation is illustrated as a Venn diagram in Figure \ref{fig:venn_diag} where the intersection of both phenomena would be a strongly exotic $H$-slice trace surgery.
It is very interesting that we can individually get both sides of this Venn diagram, but not both at the same time.
This would be a counterexample to the $H$-slice version of Conjecture \ref{conj:slice_akb_kirb}.
Such a counterexample could help clarify how to find a counterexample to Conjecture \ref{conj:slice_akb_kirb} and successfully implement the Manolescu-Piccirillo approach to SPC4.

Constructing a strongly exotic $H$-slice trace surgery seems to be a difficult problem.
Qin considered the analogous surgery with framed sliceness as a possible method to construct an exotic $\nCPbar{}$ \cite{qin2023rbg}.
It is not difficult to construct a strongly exotic framed slice trace surgery with non-zero framing using an exotic pair of traces.
To make progress towards constructing a strongly exotic $H$-slice trace surgery, we need stronger obstructions to $H$-sliceness.
Manolescu-Marengon-Piccirillo gave the first example of an obstruction to $H$-sliceness that can be used to exhibit an exotically $H$-slice knot \cite{rel_gen_indef_manifold}.
Despite the author's best efforts, he was unable to use this obstruction to construct a strongly exotic $H$-slice trace surgery.
It may be possible these attempts were in vain and this obstruction cannot be used to exhibit strongly exotic $H$-slice trace surgeries.
\begin{conj}
If there is a zero surgery homeomorphism $\zsg{}$ and $K$ is $H$-slice in a symplectic $4$-manifold $X$ with $b^+(X) = 3 \ mod \ 4$, then $-K'$ does not bound a disk $\Delta$ in any symplectic $4$-manifold $X'$ with $b^+(X') = 3 \ mod \ 4$, $[\Delta]^2 \ge 0$, and $[\Delta] \neq 0$.
\end{conj}
The above conjecture asserts that the Manolescu-Marengon-Piccirillo obstruction cannot detect strongly exotic $H$-slice trace surgeries.
A counterexample would of course exhibit a strongly exotic $H$-slice trace surgery.
There are other directions towards constructing a strongly exotic $H$-slice trace surgery.
Stronger obstructions to $H$-sliceness and other constructions of exotic $H$-slice trace surgeries would be tremendously insightful.
In Section \ref{sec:constr}, we showed that the $-5_2$ knot is the smallest non-trivial knot $H$-slice in the $K3$ surface.
It would be interesting if this differentiates the $K3$ surface from other homotopy $K3$ surfaces.
\begin{prob}
Exhibit a homotopy $K3$ surface that the $-5_2$ knot is not $H$-slice in.
\end{prob}
%

%
\subsection{Knot traces}
There are several interesting avenues to follow up these results regarding knot traces.
The key technical result underlying this paper is the relationship between annulus twisting a knot and a torus surgery on the knot's trace.
This allowed us to do torus surgeries on a knot trace in a nice way.
In Section \ref{sec:stein} on Stein traces, we first observed how we can use Luttinger surgeries to realize our exotic traces as exotic Stein fillings.
We then used the contact annulus twisting construction to upgrade these from exotic Stein fillings to exotic Stein traces of Legendrian knots.
Theorem \ref{thm:ann_twi_log_trans} suggests it should be possible to reconcile these two constructions of exotic Stein fillings.
\begin{conj}
The analogue of Theorem \ref{thm:ann_twi_log_trans} holds for contact annulus twists and Luttinger surgeries of Stein traces.
\end{conj}
Proving this would further validate the philosophy of this paper: apply $4$-manifold techniques to the study of knot traces.
It would be desirable to apply other surgery constructions of $4$-manifolds to knot traces.
In particular, it would be interesting to find a good way to do a Fintushel-Stern knot surgery on knot traces.
\begin{prob}
Find a knot trace $X_n(K)$ and a Fintushel-Stern knot surgery on $X_n(K)$ so that the result is a different knot trace $X_n(K')$.
\end{prob}
One limitation of using torus surgeries to study knot traces is the need for a self intersection zero torus.
This makes it harder to apply the techniques of this paper to knot traces with non-zero framing as these have no non-trivial homology classes for such a torus.
This leaves open whether the results of this paper hold for traces with non-zero framings.
\begin{prob}
Construct exotic traces with non-zero framings that have the same properties as those constructed in this paper.
\end{prob}
As far as the author is aware, all prior examples of exotic traces are constructed via a cork twist.
These use the Akbulut cork and its boundary involution to get exotic pairs of traces.
The evolution from pairs to infinitely many exotic traces is analogous to the evolution of involutive corks to infinite order corks.
In fact, our construction of exotic traces was directly inspired by Gompf's construction of infinite order corks \cite{inf_ord_cor}.
We leave with our final question whether this inspiration can be made explicit.
\begin{prob}
Exhibit an infinite order cork $(C,f)$ and an embedding $C \subset X_0(J^m_0)$ so that the cork twist by $f^k$ results in $X_0(J^m_k)$.
\end{prob}

\printbibliography
\end{document}